\documentclass[12pt]{article}
\usepackage{amssymb,eucal,latexsym}
\usepackage{amsthm,amscd,amsbsy}
\usepackage[leqno]{amsmath}

\usepackage{geometry}

{\par\medskip \color{blue}{ID: \  }}{\hfill{$\Box$}\color{black}}

{\par\medskip \color{blue}{GD: \  }}{\hfill{$\Box$}\color{black}}

\parskip=\medskipamount
\theoremstyle{plain}

\newtheorem{theorem}{Theorem}[section]
\newtheorem{lm}[theorem]{Lemma}
\newtheorem{exa}[theorem]{Example}

\newtheorem{cor}[theorem]{Corollary}
\newtheorem{pro}[theorem]{Proposition}

\newtheorem{nota}[theorem]{Notation}

\newtheorem{fact}[theorem]{Fact}

\theoremstyle{definition}
\newtheorem{defi}[theorem]{Definition}

\newcommand{\symdiff}{\triangle}
\newcommand{\timplies}{\text{ implies }}

\newcommand{\card}[1]{\ensuremath{|#1|}}

\newcommand{\df}{\ensuremath{\overset{\mathrm{df}}{=}}}

\newcommand{\tiff}{if and only if \ }
\newcommand{\Iff}{\Longleftrightarrow}

\newcommand{\Implies}{\Rightarrow}
\newcommand{\klam}[1]{\ensuremath{\langle #1 \rangle}}
\newcommand{\set}[1]{\ensuremath{\{#1\}}}

\newcommand{\restrict}{\upharpoonright}

\newcommand{\ord}{\operatorname{ord}}

\newcommand{\wlg}{w.l.o.g.\ }

\newcommand{\RC}{\operatorname{RC}}
\newcommand{\RO}{\operatorname{RO}}
\newcommand{\CR}{\operatorname{CR}}

\newcommand{\CO}{\operatorname{CO}}

\def\p{\varphi}

\def\b{\beta}
\def\d{\delta}

\def\g{\gamma}
\def\GA{\Gamma}

\def\LAM{\Lambda}

\def\lra{\longrightarrow}

\def\sbe{\subseteq}

\def\stm{\setminus}
\def\ems{\emptyset}
\def\nes{\neq\emptyset}

\def\cuk{\,\check{}\,}

\def\unlb{\underline{B}}

\def\ex{\exists}
\def\fa{\forall}

\def\we{\wedge}

\def\bv{\bigvee}

\def\ap{^{\prime}}
\def\inv{^{-1}}
\def\st{\ |\ }

\def\llx{\ll_{\rho}}
\def\lle{\ll_{\eta}}

\def\1{{\bf 1}}
\def\2{\mbox{{\bf 2}}}
\def\3{\mbox{{\bf 3}}}

\def\AA{{\cal A}}
\def\BB{{\cal B}}
\def\CC{{\cal C}}

\def\DD{{\cal D}}

\def\FF{{\cal F}}

\def\LL{{\cal L}}

\def\PP{{\cal P}}

\def\TT{{\cal T}}
\def\UU{{\cal U}}
\def\VV{{\cal V}}
\def\WW{{\cal W}}

\def\HC{{\bf HC}}
\def\DHC{{\bf DHC}}
\def\HLC{{\bf HLC}}
\def\DHLC{{\bf DHLC}}

\def\ZHLC{{\bf ZHLC}}

\def\Bool{{\bf Bool}}
\def\Stone{{\bf Stone}}

\def\int{\mbox{{\rm int}}}

\def\cl{\mbox{{\rm cl}}}

\def\NNN{\ensuremath{\mathbb{N}}}
\def\BBBB{\mathbb{B}}
\def\RRRR{\mathbb{R}}
\def\NNNN{\mathbb{N}}

\def\IIII{\mathbb{I}}

\def\neset{\neq\emptyset} 

\def\tcx{t_X^C}
\def\tcy{t_Y^C}
\def\tcx0{t_{(X,X_0)}}
\def\tcy0{t_{(Y,Y_0)}}

\def\DLC{{\bf DLC}}

\def\bU0{\bar{U}=(U^0,(U^i,U^{ci})_{i\in\omega})}

\def\bV0{\bar{V}=(V^0,(V^i,V^{ci})_{i\in\omega})}

\title{{\LARGE\bf On dimension and weight  of a local \\ contact algebra}\\
\vspace{0.5cm} {\large\bf G. Dimov, E. Ivanova-Dimova and I.
D\"{u}ntsch }\thanks{The authors gratefully acknowledge support by
the Bulgarian National Fund of Science, contract
DN02/15/19.12.2016.}
\\
\vspace{0.2cm} {\footnotesize\rm Faculty of Math. and Informatics,
Sofia University,} {\footnotesize\rm 5 J. Bourchier Blvd., 1164
Sofia, Bulgaria}\\
{\footnotesize\rm School of Mathematics and Computer Science, Fujian Normal University, Fuzhou, China,}
{\footnotesize\rm and Dept. of Computer Science, Brock University,}
 {\footnotesize\rm St. Catharines, ON, L2S 3A1, Canada}
}

\author{}

\date{}

\begin{document}

\maketitle

\vspace{-15mm}

\begin{abstract}
\noindent As proved in \cite{D-AMH1-10}, there exists a duality
$\LAM^t$ between the category  $\HLC$  of  locally compact
Hausdorff spaces and  continuous maps, and the category $\DHLC$ of
complete local contact algebras and appropriate morphisms between
them. In this paper, we introduce the notions of weight $w_a$ and of dimension $\dim_a$
 of a local contact algebra, and we prove that if $X$ is a
locally compact Hausdorff space then $w(X)=w_a(\LAM^t(X))$, and
if, in addition, $X$ is normal, then $\dim(X)=\dim_a(\LAM^t(X))$.
\end{abstract}



\footnotetext[3]{{\footnotesize {\em E-mail addresses:}
gdimov@fmi.uni-sofia.bg, elza@fmi.uni-sofia.bg,
ivo@duentsch.net}}

\section{Introduction}

According to Stone's famous duality theorem \cite{Stone}, the
Boolean algebra $\CO(X)$ of all clopen (= closed and open) subsets
of a zero-dimensional compact Hausdorff space $X$ carries the
whole information about the space $X$, i.e. the space $X$ can be
reconstructed from $\CO(X)$, up to homeomorphism. It is natural
to ask whether the Boolean algebra $\RC(X)$ of all regular closed
subsets of a compact Hausdorff space $X$ carries the full
information about the space $X$ (see Example \ref{rct} below for $\RC(X)$). It is well known that the answer
is $``$No". For example, the Boolean algebras  of all regular
closed subsets of the unit interval $\IIII$ (with its natural
topology) and the absolute $a\IIII$ of $\IIII$ (i.e. the Stone
dual of $\RC(\IIII)$) are isomorphic but $\IIII$ and $a\IIII$ are
not homeomorphic because $\IIII$ is connected and $a\IIII$ is not (see, e.g., \cite{PW} for absolutes).
Suppose that $\HC$ is the category of compact Hausdorff spaces and continuous maps, and that $X$ is a compact Hausdorff space. As shown by H. de Vries \cite{deV}, all information about the space $X$ is contained in the pair $\klam{\RC(X),\rho_X}$, where $\rho_X$ is a binary relation on $\RC(X)$ such that for all $F,G\in \RC(X)$,
\begin{gather*}
F\rho_X G  \text{ \tiff}  F\cap G\nes.
\end{gather*}
In order to describe abstractly  the pairs $\klam{\RC(X),\rho_X}$, he introduced the notion of {\em compingent
Boolean algebra}, and he proved that there exists a duality between the category $\HC$ and the category $\DHC$ of complete compingent Boolean algebras and appropriate morphisms between them.

Subsequently, Dimov \cite{D-AMH1-10} extended de Vries' duality from the category $\HC$ to the category $\HLC$ of locally compact
Hausdorff spaces and continuous maps, and, on the base of this result, he  also obtained an extension of Stone's duality from the category $\Stone$ of compact zero-dimensional Hausdorff spaces and continuous maps to  the category $\ZHLC$ of zero-dimensional locally compact Hausdorff spaces and continuous maps (see \cite{D-PMD12,D-a0903-2593}).

The paper \cite{D-AMH1-10} has its precursor in results by P. Roeper \cite{Roeper}, who showed that all information about a locally compact Hausdorff space $X$ is contained in the triple
\begin{gather*}
\klam{\RC(X),\rho_X,\CR(X)},
\end{gather*}
 where
$\CR(X)$ is the set of all compact regular closed subsets of $X$. In order to describe abstractly the triples
$\klam{\RC(X),\rho_X,\CR(X)}$, he introduced the notion of {\em region-based topology}, and he proved that -- up to homeomorphisms, respectively, isomorphisms -- there exists a bijection between the class of all locally compact Hausdorff spaces and the class of all complete region-based topologies. The duality theorem proved in \cite{D-AMH1-10} says that there exists a duality $\LAM^t$ between the category $\HLC$
and  the category $\DHLC$ of all complete region-based topologies and appropriate morphisms between them.
Note that $$\LAM^t(X)\df \klam{\RC(X),\rho_X,\CR(X)},$$ for every locally compact Hausdorff space $X$.

In \cite{DV1}, the general notion of {\em Boolean contact algebra} was introduced and, accordingly, $``$compingent Boolean algebras"\/ were called $``$normal Boolean contact algebras" (abbreviated as NCAs), and $``$region-based topologies"\/ were called $``$local contact
Boolean algebras" (abbreviated as LCAs). Typical examples of Boolean contact algebras are the pairs $$\klam{\RC(X),\rho_X},$$ where $X$ is an arbitrary topological space.
We will even  use a more general notion, namely, the notion of a {\em Boolean precontact algebra}, introduced by D\"{u}ntsch and Vakarelov in \cite{DUV}.

Having a duality $\LAM^t$ between the categories $\HLC$ and
$\DHLC$, it is natural to look for the algebraic expressions dual to
topological properties of locally compact Hausdorff spaces. It
is easy to find such an expression for the property of
$``$connectedness"  even for arbitrary topological spaces, see \cite{Biacino-and-Gerla-1996}. Namely, a Boolean contact  algebra $\klam{B,C}$ is said to be {\em connected}\/ if $a\neq 0,1$
implies that $aCa^*$; here, $a^*$ is the Boolean complement of $a$. It was proved in
\cite{Biacino-and-Gerla-1996} that for a topological space $X$, the  Boolean contact algebra $\klam{\RC(X),\rho_{X}}$ is connected  \tiff  the space $X$ is connected.

In this paper we introduce the notions of {\em dimension of a precontact algebra}\/ and {\em weight of
a local contact algebra}, and prove that

\begin{enumerate}
\item The weight of a locally compact Hausdorff space $X$ is equal to
the weight of the local contact algebra $\LAM^t(X)$ (Theorem \ref{weightpiw}), and
\item The {\v C}ech--Lebesgue dimension of a normal   $T_1$-space $X$ is equal to the dimension of the
Boolean contact algebra $\klam{\RC(X),\rho_X}$ (Theorem \ref{thdim}). In particular, the {\v C}ech--Lebesgue dimension of a normal locally
compact Hausdorff space $X$ is equal to the dimension of the local
contact algebra $\LAM^t(X)$ (Corollary \ref{thdimcor}).
\end{enumerate}

One cannot define a notion of dimension for Boolean algebras corresponding to the topological notion of dimension via de Vries' or Dimov's dualities because for all positive natural numbers $n$ and $m$, the Boolean algebras $\RC(\RRRR^n)$  and $\RC(\RRRR^m)$ are isomorphic (see Birkhoff \cite[p.177]{bir48}) but, clearly, for $n\neq m$, $\dim(\RRRR^n)\neq \dim(\RRRR^m)$. Also, one cannot define an adequate (in the same sense)  notion of weight for Boolean algebras because, for example, the Boolean algebras $\RC(\IIII)$  and $\RC(a\IIII)$ are isomorphic but $w(\IIII)=\aleph_0<2^{\aleph_0}=w(a\IIII)$ (see \cite[Chapter VI, Problem 234(a)]{AP}).

The paper is organized as follows.  Section 2 contains all preliminary facts and definitions which are used in this paper. In Section 3, we introduce and study the notion of dimension of a precontact algebra.  Here we prove Theorem \ref{thdim} and Corollary \ref{thdimcor}, mentioned above. It is shown as well that the dimension of a normal contact algebra is equal to the dimension of its NCA-completion (see \cite{D-AMH2-10,D-a0903-2593} for this notion), that the dimension of any NCA of the form $\klam{B,\rho_s}$ (where $\rho_s$ is the smallest  contact relation on $B$) is equal to zero (as it should be),  and that the dimension of every relative LCA of an LCA $\klam{B,\rho,\BBBB}$ is smaller or equal to $\dim_a(\klam{B,\rho,\BBBB})$. Recall that L. Heindorf (cited in \cite{monk14}) introduced the notion of $\AA$-dimension for Boolean algebras, where $\AA$ is an arbitrary non-empty class  of Boolean algebras. There is, however, no connection between the topological notion of dimension and the notion of $\AA$-dimension, so that his investigations are in a different direction from those carried out here.

In Section 4, we introduce and study the notion of weight of a local contact algebra.  Here we prove Theorem \ref{weightpiw}, mentioned above. We show as well that the weight of a local contact algebra is equal to the weight of its LCA-completion (see \cite{D-AMH2-10,D-a0903-2593} for this notion), find an algebraic analogue of Alexandroff-Urysohn theorem for bases (\cite[Theorem 1.1.15]{E}), describe the LCAs whose dual spaces are metrizable, and characterize the LCAs whose dual spaces are zero-dimensional. Furthermore, for  a  dense Boolean subalgebra $A_0$ of a Boolean algebra $A$, we construct an NCA $\klam{A,\rho}$ such that $w_a(\klam{A,\rho})=|A_0|$,  and if $A$ is complete, then its dual space is homeomorphic to the Stone dual of $A_0$.

 In Section 5, we discuss the relationship between  algebraic density and  algebraic weight, introduce the notion of a $\pi$-{\em semiregular space}, and show that if $X$ is $\pi$-semiregular then
 $\pi w(X)$ is equal to the density of the Boolean algebra $\RC(X)$. Finally,  for every $\pi$-semiregular space $X$ with $\pi w(X)\ge\aleph_0$, we prove that  there exists a zero-dimensional compact Hausdorff space $Y$ with $w(Y)=\pi w(X)$ such that the Boolean algebras $\RC(X)$ and $\RC(Y)$ are isomorphic.

 The results from Sections 4 and 5 are from the arXiv-paper \cite{D-a0903-2593}.

\section{Preliminaries}


\subsection{Notation and first definitions}\label{sec:not}

Suppose that $\klam{P, \leq, 0}$ is a partially ordered set with smallest element $0$. If $M \subseteq P$, then $M^+ \df M \setminus \set{0}$. $M$ is called \emph{dense in $P$}, if for all $a \in P^+$ there is some $b \in M^+$ such that $b \leq a$.

A {\em join-semilattice} is a partially ordered set having all finite non-empty joins.

We denote by  ${\NNNN}$ the set of all non-negative integers, by $\NNNN^-$ the set $\NNNN\cup\{-1\}$, by $\mathbb{R}$ the real line (with its
natural topology), and by $\mathbb{I}$ the subspace $[0,1]\ (\df \{x\in\mathbb{R}\st 0\le x\le 1\})$ of $\mathbb{R}$.

The power set of a set $X$ is denoted by $2^X$; we implicitly suppose that $2^X$ is a Boolean algebra under the set operations.

Throughout, $\klam{B, \land, \lor, {}^*, 0, 1}$ will denote a Boolean algebra unless indicated otherwise; we do not assume that $0 \neq 1$. With some abuse of language, we shall usually identify algebras with their universe, if no confusion can arise.

If $B$ is a Boolean algebra and $b \in B^+$, we let $B_b$  be the relative algebra of $B$ with respect to $b$ \cite[Lemma 3.1.]{kop89}.

If $a,b \in B$, then $a \symdiff b$ denotes the symmetric difference of $a$ and $b$, i.e. $a \symdiff b \df  (a \land b^* ) \lor (b \land a^*)$. It is well known that $a \symdiff b = 0$ \tiff $a = b$.



Throughout, $(X,\TT)$ will be a topological space. If no confusion can arise, we shall just speak of $X$.
We denote by $\CO(X)$ the set of all clopen (= closed and open) subsets of $X$; clearly, $\klam{\CO(X),\cup,\cap,\stm,\ems, X}$ is a Boolean algebra.
A subset $F$ of $X$ is called {\em regular closed}\/ (resp., {\em regular open}\/) if $F=\cl(\int(F))$ (resp., $F=\int(\cl(F))$). We let $\RC(X)$ (resp., $\RO(X)$) be the set of all regular closed (resp., regular open) subsets of $X$. The space $X$ is called {\em semiregular} if $\RO(X)$ is an open base for $X$, or, equivalently, if $\RC(X)$ is a closed base for $X$.


A \emph{cover} of a set $X$ is a family $\AA$ of subsets of $X$
for which $\bigcup \AA = X$. If $\AA, \BB$ are covers of $X$, then
$\BB$ is a \emph{refinement of $\AA$}, if for every $B \in \BB$
there is some $A \in \AA$ such that $B \subseteq A$. A cover $\BB
\df  \set{B_i\st i \in I}$ is a \emph{shrinking of $\AA \df  \set{A_i\st i
\in I}$} if $B_i \subseteq A_i$ for all $i \in I$. If $\AA \df
\set{A_i\st i \in I} \subseteq 2^X$, a family $\BB \df  \set{B_i\st i \in
I} \subseteq 2^X$ is called a \emph{swelling of $\AA$}, if $A_i
\subseteq B_i$ for all $i \in I$, and for all $k\in\NNNN^+$ and
$i_1, \ldots, i_k \in I$,
\begin{gather*}
A_{i_1} \cap \ldots \cap A_{i_k} = \ems \Iff B_{i_1} \cap \ldots
\cap B_{i_k} = \ems.
\end{gather*}

A  cover (refinement, shrinking, swelling) of a topological space
$X$ is called \emph{open (regular open, closed, regular closed)}
if all of its members are open (regular open, closed, regular
closed) subsets of $X$.

If $X$ is a set and $\mathcal A \subseteq 2^X$, the \emph{order of
$\mathcal A$} is defined as
\begin{gather*}
\ord \AA \df
\begin{cases}
n, &\text{if } n = \max\set{m\in\NNNN^-\st (\exists A_1, \ldots,
A_{m+1} \in \AA)(\bigcap_{i=1}^{m+1} A_i \neq \ems)},\\ \infty,
&\text{if such $n$ does not exist}.
\end{cases}
\end{gather*}

It follows that if $\ord \AA = n$, then the intersection of every
$n+2$ distinct elements of $\AA$ is empty. Also, $\ord \AA= -1$
 \tiff  $\AA=\{\ems\}$, and $\ord \AA= 0$  \tiff  $\AA$ is a disjoint
family of subsets of $X$ which are not all empty.

The \emph{{\v C}ech--Lebesgue dimension} of a topological space $X$, denoted by $\dim(X)$, is defined in layers, see e.g.,\cite{eng_dim}.
 Suppose that $n \in \NNNN^-$.
\begin{quote}
\begin{enumerate}
\renewcommand{\theenumi}{\ensuremath{(CL\arabic{enumi})}}
\item If every finite open cover of $X$ has a finite open refinement of order at most $n$, then $\dim(X) \leq n$. \label{cle1}
\item If $\dim(X) \leq n$ and $\dim(X) \not\leq n-1$, then $\dim(X) = n$. \label{cle2}
\item If $n \lneq \dim(X)$ for all $n \in \NNNN^-$, then $\dim(X) = \infty$. \label{cle3}
\end{enumerate}
\end{quote}
Observe that $\dim(X) = -1$ \tiff $X = \emptyset$.

If $\CC$ is a category, we denote by $\card{\CC}$ the class of all objects of the category $\CC$, and by $\CC(A,B)$ the set of all $\CC$-morphisms between the $\CC$-objects $A$ and $B$.

For unexplained notation we invite the reader to consult \cite{kop89} for Boolean algebras, \cite{AHS} for category theory, and
 \cite{E} for topology.



\subsection{Boolean (pre)contact algebras}


In this paper we work mainly with Boolean algebras with  supplementary structures on them. In all cases, we will say that the corresponding
structured Boolean algebra is {\em complete}\/ if the underlying Boolean algebra is complete.

\begin{defi}\label{precontact}{\rm (\cite{DUV})}
\rm
A {\em Boolean precontact algebra}, or, simply, {\em precontact algebra}\/ (PCA) (originally, {\em Boolean proximity algebra} \cite{DUV}), is a structure $\klam{B,C}$, where $B$ is a Boolean algebra, and $C$ a binary relation on $B$, called a precontact relation, which satisfies the following axioms:

\begin{enumerate}
\renewcommand{\theenumi}{\ensuremath{(C\arabic{enumi})}}
    \item If $aCb$ then $a\not=0$ and $b\not=0$.\label{c2}
     \item $aC(b\vee c)$  \tiff  $aCb$ or $aCc$; $(a\vee b)Cc$ \tiff
$aCc$ or $bCc$.\label{c4}
\end{enumerate}

\noindent
 Two precontact algebras\/ $\klam{B,C}$ and\/
$\klam{B_1,C_1}$ are said to be\/ {\em PCA-isomorphic}
(or, simply,\/ {\em isomorphic}) if there exists a {\em PCA-isomorphism} between them, i.e., a Boolean
isomorphism $\varphi: B\longrightarrow B_1$ such that, for every
$a,b\in B$, $aCb$ iff $\varphi(a)C_1 \varphi(b)$.
\end{defi}

The notion of a precontact algebra was defined independently (and in a completely different form) by  S. Celani \cite{Celani}.
A duality theorem for precontact algebras was obtained in \cite{DDV} (see also \cite{DV-LN06,DV3}).

\begin{defi}\label{conalg}
\rm
A PCA $\klam{B,C}$ is called a {\it Boolean contact algebra}\/  \cite{DV1} or, briefly, a {\it contact algebra} (CA), if it satisfies the following additional axioms for all $a,b \in B$:

\begin{enumerate}
\renewcommand{\theenumi}{\ensuremath{(C\arabic{enumi})}}
\setcounter{enumi}{2}
    \item If $a\not= 0$ then $aCa$.\label{c1}
    \item $aCb$ implies $bCa$.\label{c3}
\end{enumerate}
The relation $C$ is  called a \emph{contact relation}.
As usual, if $a\in B$,  we set $$C(a)\df \{b\in B\st aCb\}.$$
\end{defi}


We shall consider two more properties of contact algebras:

\begin{enumerate}
\renewcommand{\theenumi}{\ensuremath{(C\arabic{enumi})}}
\setcounter{enumi}{4}
    \item If $a(-C)b$ then $a(-C)c$ and $b(-C)c^*$ for some
$c\in B$.\label{c5}
    \item  If $a\not= 1$ then there exists $b\not= 0$ such that
$b(-C)a$.  \label{c6}
\end{enumerate}

A contact algebra $\klam{B,C}$ is called a {\it  Boolean normal contact  algebra}\/ or, briefly, {\it  normal contact algebra}
(abbreviated as NCA) \cite{deV,F} if it satisfies \ref{c5} and \ref{c6}.

\noindent
The notion of normal contact algebra was introduced by Fedorchuk \cite{F} under the name of {\em Boolean $\d$-algebra}\/ as an equivalent
expression of the notion of {\em compingent Boolean algebra} of de Vries (see the definition below). We call such algebras ``normal contact algebras'' because they form a subclass of the class of contact algebras which naturally arise as canonical algebras
in normal Hausdorff spaces (see \cite{DV1}).

Axiom \ref{c6} is an extensionality axiom since a CA $\klam{B,C}$ satisfies \ref{c6}  \tiff $(\forall a,b \in B)[C(a) = C(b) \timplies a = b]$
%
%
(see \cite[Lemma 2.2]{DV1}). Keeping this in mind, we call a CA $\klam{B,C}$  an {\em extensional contact algebra}\/ (abbreviated as ECA) if it satisfies \ref{c6}. This notion was introduced in \cite{DW} under the name of {\em Boolean contact algebra}, and a representation theorem for ECAs was proved there.

Note that if $0\neq 1$, then \ref{c2} follows from the axioms \ref{c4}, \ref{c3}, and \ref{c6}.

\begin{defi}\label{def:ll}
\rm
For a PCA $\klam{B,C}$, we define a binary relation  $``\ll_C $'' on $B$, called {\em non-tangential inclusion},  by
\begin{gather*}
\ a \ll_C b \text{ \tiff} a(-C)b^*.
\end{gather*}
Here, $-C$ is the set complement of $C$ in $B \times B$. If $C$ is understood, we shall  simply write
``$\ll$'' instead of ``$\ll_C$''.
\end{defi}

The relations $C$ and $\ll$ are inter-definable. For example,
normal contact algebras may be equivalently defined -- and exactly
in this way they were introduced under the name of {\em
compingent Boolean algebras}\/ by de Vries in \cite{deV} -- as a pair consisting
of a Boolean algebra $B$ and a binary relation $\ll$ on $B$ satisfying the following axioms:

\begin{quote}
\begin{enumerate}
\renewcommand{\theenumi}{($\ll$\arabic{enumi})}
\item $a \ll b \timplies a \leq b$. \label{di1}
\item $0\ll 0$. \label{di2}
\item  $a\leq b\ll c\leq t$ implies $a\ll t$. \label{di3}
\item $a\ll c$ and $b\ll c$ implies $a\vee b\ll c$. \label{di4}
\item If  $a\ll c$ then $a\ll b\ll c$  for some $b\in B$. \label{di5}
\item  If $a\neq 0$ then there exists $b\neq 0$ such that $b\ll
a$. \label{di6}
\item$a\ll b$ implies $b^*\ll a^*$. \label{di7}
\end{enumerate}
\end{quote}

Indeed, if $\klam{B,C}$ is an NCA, then the relation $\ll_C$ satisfies the axioms \ref{di1} -- \ref{di7}.  Conversely,
having
 a pair $\klam{B, \ll}$, where $B$ is a Boolean algebra and $\ll$ is a binary relation  on $B$ which satisfies \ref{di1} -- \ref{di7}, we define a relation
$C_\ll$
by $aC_\ll b$ \tiff $a\not\ll b^*$; then $\klam{B,C_\ll}$ is an NCA. Note that the axioms \ref{c5} and \ref{c6} correspond to \ref{di5} and to \ref{di6}, respectively. It is easy to see that contact algebras could be equivalently defined as a
pair of a Boolean algebra $B$ and a binary relation $\ll$ on $B$
subject to the  axioms \ref{di1} -- \ref{di4} and \ref{di7}; then, clearly,  the relation $\ll$ also
satisfies  the axioms

\smallskip

\noindent($\ll$2') $1\ll 1$;\\
($\ll$4') ($a\ll c$ and $b\ll c$) implies $(a\vee b)\ll c$.

 It is not difficult to see that precontact algebras could be equivalently defined as a
pair of a Boolean algebra $B$ and a binary relation $\ll$ on $B$
subject to the  axioms ($\ll$2), ($\ll$2'), ($\ll$3), ($\ll$4) and ($\ll$4').

A mapping $\p$ between two contact algebras $\klam{B_1,C_1}$ and $\klam{B_2,C_2}$ is called a \emph{CA-morphism} (\cite{DDV}), if $\p: B_1 \lra B_2$ is a Boolean homomorphism, and
$\p(a) C_2 \p(b)$ implies $aC_1 b$, for any $a,b\in B_1$.
Note that $\p:\klam{B_1,C_1}\lra \klam{B_2,C_2}$ is a CA-morphism \tiff  $a\ll_{C_1} b$ implies $\p(a) \ll_{C_2} \p(b)$, for any $a,b\in B_1$.
 (Thus, a CA-morphism is a structure preserving morphism between $\klam{B_1,\ll_1}$ and $\klam{B_2,\ll_2}$ in the sense of first order logic.)
Two CAs $\klam{B_1,C_1}$ and $\klam{B_2,C_2}$ are {\em CA-isomorphic}  \tiff  there exists a
bijection $\p:B_1\lra B_2$ such that $\p$ and $\p\inv$ are CA-morphisms.

The following assertion may be worthy of mention:

\begin{pro}\label{pro:injective}
If $\klam{B_1,C_1}$ and $\klam{B_2,C_2}$ are CAs, $\p: B_1 \lra B_2$ is a Boolean homomorphism and
$\p$ preserves the contact relation $C_1$ (i.e., $aC_1 b$ implies $\p(a) C_2 \p(b)$, for all $a,b\in B_1$), then $\p$  is an injection.
\end{pro}
\begin{proof}
Assume that $\p$ is not injective. Then, there are $a,b \in B_1$ such that $a \neq b$ and $\p(a) = \p(b)$; hence, $c \df  a \symdiff b \neq 0$, and $\p(c) = 0$. By \ref{c1}, $c C_1 c$, and the fact that $\p$ preserves $C_1$ implies that $\p(c) C_2 \p(c)$, i.e. $0C_2 0$. This contradicts \ref{c2}.
\end{proof}

The most important ``concrete'' example of a CA is given by the regular closed sets of an arbitrary topological space.

\begin{exa}\label{rct}
\rm Let $(X,\TT)$ be a  topological space. The collection $\RC(X,\TT)$  becomes a complete Boolean algebra
$\klam{\RC(X,\TT),0,1,\we,\vee,{}^*}$ under the following operations:
\begin{align*}
F\vee G &\df F\cup G, & F\we G &\df \cl(\int(F\cap G)), & F^* &\df  \cl(X\stm F), & 0 &\df  \emptyset, & 1 &\df  X.
\end{align*}
 The infinite operations are given by the  formulas
 \begin{align*}
 \bigvee\{F_\g\st \g\in\GA\} &\df  \cl(\bigcup_{\g\in\GA}F_\g)\
(=\cl(\bigcup_{\g\in\GA}\int(F_\g))=\cl(\int(\bigcup_{\g\in\GA}F_\g))), \\
\bigwedge\{F_\g\st \g\in\GA\} &\df  \cl(\int(\bigcap\{F_\g\st
\g\in\GA\})),
 \end{align*}

Define a  relation $\rho_{(X,\TT)}$ on $\RC(X,\TT)$ by setting, for each $F,G\in \RC(X,\TT)$,
 \begin{gather*}
 F \rho_{(X,\TT)}G \mbox{  \tiff  } F\cap G\neq \ems.
 \end{gather*}
 Clearly, $\rho_{(X,\TT)}$ is a contact relation, called the {\em standard contact relation of $(X, \TT)$}. The  complete CA $\klam{\RC(X,\TT),\rho_{(X,\TT)}}$  is called a {\em standard contact algebra}. If no confusion can arise, we shall usually write simply $\RC(X)$ instead of $\RC(X,\TT)$, and $\rho_X$ instead of $\rho_{(X,\TT)}$. Note that, for $F,G\in \RC(X)$,
 \begin{gather*}
 F\ll_{\rho_X}G \mbox{  \tiff  } F\sbe\int_X(G).
 \end{gather*}
Thus, if $(X,\TT)$ is a normal Hausdorff space then the standard contact algebra $$\klam{\RC(X,\TT),\rho_{(X,\TT)}}$$ is a complete NCA.

Instead of looking at regular closed sets, we may, equivalently, consider regular open sets. The collection $\RO(X)$ of regular open sets becomes a complete Boolean algebra by setting
\begin{align*}
  U\vee V &\df  \int(\cl(U\cup V)), & U \land V &\df  U \cap V, & U^* &\df  \int(X\stm U), & 0 &\df  \emptyset, & 1 &\df  X,
\end{align*}
and
\begin{align*}
\bigwedge_{i\in I} U_{i}& \df \int(\cl(\bigcap_{i\in I}U_{i}))\ (= \int(\bigcap_{i\in
I}U_{i})), & \bigvee_{i\in I} U_{i} &\df \int(\cl(\bigcup_{i\in I}U_{i})),
\end{align*}
see \cite[Theorem 1.37]{kop89}. We define a contact relation $D_X$ on $\RO(X)$ as follows:
\begin{gather*}
U D_X V \mbox{  \tiff  } \cl (U)\cap \cl (V)\neset.
\end{gather*}
Then $\klam{\RO(X),D_X}$ is a complete CA.

The contact algebras $\klam{\RO(X),D_X}$ and $\klam{\RC(X),\rho_X}$ are CA-isomorphic via the mapping
$\nu: \RO(X) \lra \RC(X)$ defined by the formula $\nu(U)\df \cl(U)$,
%
%
for every $U\in \RO(X)$.
\end{exa}

\begin{exa}\label{extrcr}
\rm Let $B$ be a Boolean algebra. Then there exist a largest and a
smallest contact relations on $B$; the largest one, $\rho_l^B$, is
defined by $$a\rho_l^B b \iff (a\neq 0\mbox{ and }b\neq 0),$$ and the
smallest one, $\rho_s^B$, by $$a\rho_s^B b \iff a\wedge b\neq 0.$$
When there is
no ambiguity, we will simply write $\rho_s$ instead of $\rho_s^B$, and $\rho_l$ instead of $\rho_l^B$.

Note that, for $a,b\in B$, $$a\ll_{\rho_s} b \iff a\le b;$$ hence
$a\ll_{\rho_s} a$, for any $a\in B$. Thus $(B,\rho_s)$ is a normal
contact algebra.
\end{exa}

\subsection{Local contact algebras}

Local contact algebras were introduced by Roeper \cite{Roeper} under the somewhat misleading name {\em region-based topologies}. Since every region-based topology is a contact algebra and also a lattice-theoretical counterpart of Leader's notion of {\em local proximity} \cite{LE}, it was suggested in \cite{DV1} to rename them to \emph{Boolean local contact  algebras}.

\begin{defi}\label{locono}{\rm \cite{Roeper}}
\rm A system $\klam{B, \rho, \BBBB}$ is called a \emph{Boolean local contact  algebra}\/ or, briefly,
\emph{local contact algebra} (abbreviated as LCA or as LC-algebra)   if $B$ is a Boolean algebra, $\rho$ is a contact relation on $B$, and $\BBBB$ is a not necessarily proper ideal of $B$ satisfying the following axioms:

\begin{quote}
\begin{enumerate}
\renewcommand{\theenumi}{(LC\arabic{enumi})}
\item If $a\in\BBBB$, $c\in B$ and $a\ll_\rho c$ then $a\ll_\rho b\ll_\rho c$ for some $b\in\BBBB$. \label{bc1}
\item If $a\rho b$ then there exists an element $c$ of $\BBBB$ such that $a\rho (c\we b)$. \label{bc2}
\item If $a\neq 0$ then there exists some $b\in\BBBB^+$ such that $b\ll_\rho a$. \label{bc3}
\end{enumerate}
\end{quote}
The elements of $\BBBB$ are called {\em bounded}, and the elements of $B\stm \BBBB$  are called  {\em unbounded}.
\end{defi} 
It may be worthy to note that it follows from a result by M. Rubin \cite{rub76}, that the first order theory of LCAs is undecidable.

Two local contact algebras $\klam{B,\rho,\BBBB}$ and $\klam{B_1,\rho_1,\BBBB_1})$ are  {\em LCA-isomorphic} if there exists a CA-isomorphism $\p:\klam{B,\rho}\lra \klam{B_1,\rho_1}$ such that, for any $a\in B$,
$\p(a)\in\BBBB_1$  \tiff  $a\in\BBBB$.

A map $\p:\klam{B,\rho,\BBBB}\lra\klam{B_1,\rho_1,\BBBB_1}$ is called an {\em LCA-embedding} if
$\p:\klam{B,\rho}\lra \klam{B_1,\rho_1}$ is a  CA--morphism such that  for any $a,b\in B$, $a\rho b$
 implies $\p(a)\rho_1\p(b)$, and
$\p(a)\in\BBBB_1$  \tiff  $a\in\BBBB$. Note that the name is justified, since, as it follows from  Proposition \ref{pro:injective}, any LCA--embedding is an injection.

If $\klam{B,\rho,\BBBB}$ is a local contact algebra and $\BBBB=B$, i.e., $\BBBB$ is an improper ideal, then $\klam{B,\rho}$ is a normal contact algebra. Conversely, any normal contact algebra $\klam{B,C}$ can be regarded as a local contact algebra of the form $\klam{B,C,B}$.

\begin{pro}\label{stanlocn}{\rm \cite{Roeper,VDDB}}
Let $X$ be a locally compact Hausdorff space. Then the triple $\klam{\RC(X),\rho_{X}, \CR(X)}$ is a complete local contact algebra.
\end{pro}

The complete LCA $\klam{\RC(X),\rho_{X}, \CR(X)}$ is called the {\em standard local contact algebra} of $X$.
%


 We will need the following notation: for every
function $\psi:\klam{B,\rho,\BBBB}\lra \klam{B\ap,\eta,\BBBB\ap}$ between two
LCAs, the function $$\psi\cuk:\klam{B,\rho,\BBBB}\lra
\klam{B\ap,\eta,\BBBB\ap}$$ is defined by
\begin{equation}\label{cukfcon}
\ \psi\cuk(a)\df\bigvee\{\psi(b)\st b\in \BBBB, b\llx a\},
\end{equation}
for every $a\in B$.

\begin{defi}\label{dhc}{\rm(\cite{D-AMH1-10})}
\rm
Let $\DHLC$ be the category whose objects are all complete LC-algebras
and whose morphisms are all functions $\p:\klam{B,\rho,\BBBB}\lra
\klam{B\ap,\eta,\BBBB\ap}$ between the objects of $\DHLC$ satisfying the following
conditions:

\smallskip

\noindent(DLC1) $\p(0)=0$;\\
(DLC2) $\p(a\we b)=\p(a)\we \p(b)$, for all $a,b\in B$;\\
(DLC3) If $a\in\BBBB, b\in B$ and $a\llx b$, then $(\p(a^*))^*\lle
\p(b)$;\\
(DLC4) For every $b\in\BBBB\ap$ there exists $a\in\BBBB$ such that
$b\le\p(a)$;\\
\noindent(DLC5) $\p(a)=\bigvee\{\p(b)\st b\in\BBBB, b\llx a\}$,
for every $a\in B$;

\medskip

{\noindent} the composition $``\diamond$" of two morphisms
$\p_1:\klam{B_1,\rho_1,\BBBB_1}\lra \klam{B_2,\rho_2,\BBBB_2}$ and
$\p_2:\klam{B_2,\rho_2,\BBBB_2}\lra \klam{B_3,\rho_3,\BBBB_3}$ of\/ $\DHLC$
is defined by the formula
\begin{equation}\label{diamcon}
\ \p_2\diamond\p_1 \df (\p_2\circ\p_1)\cuk.
\end{equation}
\end{defi}

Note that two complete LCAs are LCA-isomorphic \tiff they are $\DHLC$-isomorphic.

Let $\HLC$ (resp., $\HC$) be the category of all locally compact (resp., compact) Hausdorff spaces and all continuous maps between them. The following duality theorem for the category $\HLC$ was proved in \cite{D-AMH1-10}.

\begin{theorem}\label{lccont}{\rm(\cite{D-AMH1-10})}
The categories $\HLC$ and\/ $\DHLC$ are dually equi\-valent.
The contravariant functors which realize this duality are denoted by
$$\LAM^t:\HLC\lra\DHLC \mbox{  and } \LAM^a:\DHLC\lra\HLC.$$
The contravariant functor $\LAM^t$ is defined as follows:

$$\LAM^t(X)\df \klam{\RC(X),\rho_X,\CR(X)},$$
for every $\HLC$-object $X$, and

$$ \LAM^t(f)(G)\df \cl(f\inv(\int(G))),$$
for every $f\in\HLC(X,Y)$ and every $G\in \RC(Y)$.
\end{theorem}

In particular,  for every complete LCA $\unlb\df \klam{B,\rho,\BBBB}$ and every $X\in|\HLC|$, $\unlb$ is LCA-isomorphic to $\LAM^t(\LAM^a(\unlb))$ and $X$ is homeomorphic to $\LAM^a(\LAM^t(X))$. (We do not give here the explicit definition of the contravariant functor $\LAM^a$ because we will not use it. (It is given in \cite{D-AMH1-10}.) For our purposes here, it is enough to know that the compositions $\LAM^a\circ\LAM^t$ and $\LAM^t\circ\LAM^a$ are naturally equivalent to the corresponding identity functors (see, e.g., \cite{AHS}).)

Also, the restriction of $\LAM^t$ to the subcategory $\HC$ of the category $\HLC$ coincides with the de Vries duality functor between the category $\HC$ and the full subcategory $\DHC$ of the category $\DHLC$, having as objects all NCAs.

The next theorem shows how one can construct the dual object $\LAM^t(F)$ of
a regular closed subset $F$ of a locally compact Hausdorff space
$X$ using only $F$ and the dual object $\LAM^t(X)$ of $X$.

\begin{theorem}\label{conregclo}{\rm(\cite{D-AMH2-10})}
Let $X$ be a locally compact Hausdorff space and $F\in \RC(X)$. Let
 $B\df \RC(X)_F$ be the relative algebra of $\RC(X)$ with respect to $F$, $$\BBBB\ap\df \{G\we F\st G\in \CR(X)\}$$ and,
for every $a,b\in B$,  $a\eta b \Leftrightarrow  a\rho_X b$ (i.e., $a\eta b \Leftrightarrow a\cap
b\nes$). Then $\klam{B,\eta,\BBBB\ap}$ is LCA-isomorphic to
$\LAM^t(F)$, where $F$ is regarded as a subspace of $X$.
\end{theorem}

We will also need the following definitions and assertions.
Note that for $\g\in\GA$  and  $a\in\prod\{A_\g\st\g\in\GA\}$, $a_\g$ will denote the $\g$-th coordinate of $a$.

\begin{defi}\label{prodclca}{\rm(\cite{D-AMH2-10})}
\rm Let
$\{\klam{B_\g,\rho_\g,\BBBB_\g}\st\g\in\GA\}$ be a family of LC-algebras and
$$B\df\prod\{B_\g\st\g\in\GA\}$$ be the product of the  Boolean
algebras $\{B_\g\st\g\in\GA\}$ in the category $\Bool$ of Boolean
algebras and Boolean homomorphisms.
Let $$\BBBB\df\{b\in\prod\{\BBBB_\g\st\g\in\GA\}\st
\card{\{\g\in \GA\st b_\g\neq 0\}}<\aleph_0\}.$$
For any two points $a,b\in B$,  set
$$a\rho b\Iff\mbox{  there exists } \g\in\GA\mbox{ such that }a_\g\rho_\g b_\g.$$ Then the triple
$\klam{B,\rho,\BBBB}$ is called a {\em product of the family $\{\klam{B_\g,\rho_\g,\BBBB_\g}\st\g\in\GA\}$ of LC-algebras}. We will write
$$\klam{B,\rho,\BBBB}=\prod\{\klam{B_\g,\rho_\g,\BBBB_\g}\st\g\in\GA\}.$$
\end{defi}

\begin{theorem}\label{prodclcacat}{\rm(\cite{D-AMH2-10})}
 Let
$\{\klam{B_\g,\rho_\g,\BBBB_\g}\st\g\in\GA\}$ be a family of complete LC-algebras,
$\klam{B,\rho,\BBBB}\df\prod\{\klam{B_\g,\rho_\g,\BBBB_\g}\st\g\in\GA\}$
and $\pi_\g(a)\df a_\g,$  for every
$a\in B$ and every $\g\in\GA$.
Then the source
$\{\pi_\g:\klam{B,\rho,\BBBB}\lra\klam{B_\g,\rho_\g,\BBBB_\g}\st\g\in\GA\}$
is a product of the family
$\{\klam{B_\g,\rho_\g,\BBBB_\g}\st\g\in\GA\}$ in the category $\DHLC$.
\end{theorem}

\begin{defi}\label{clcaweo}{\rm (\cite{D-AMH2-10,D-a0903-2593})}
\rm Let $\klam{B,\rho,\BBBB}$ be an LCA and $D$ be a subset of
$\BBBB$. Then we say that $D$ is   {\em dV-dense in}
$\klam{B,\rho,\BBBB}$ if for each $a,c\in\BBBB$ such that $a\llx c$,
there exists $d\in D$ with $a\le d\le c$.
\end{defi}

\begin{fact}\label{wfact}{\rm (\cite{D-AMH2-10,D-a0903-2593})}
 If $\klam{B,\rho,\BBBB}$ is an LCA and $D$ is a subset of
$\BBBB$, then $D$ is dV-dense in $\klam{B,\rho,\BBBB}$ \tiff
for each $a,c\in\BBBB$ such that $a\llx c$, there exists $d\in D$
with $a\llx d\llx c$.
\end{fact}

\begin{defi}\label{defcompl}{\rm (\cite{D-AMH2-10,D-a0903-2593})}
\rm Let $\klam{B,\rho,\BBBB}$ be an LCA. A pair
$(\p,\klam{B\ap,\rho\ap,\BBBB\ap})$ is called an {\em LCA-completion}
of the LCA $\klam{B,\rho,\BBBB}$ if $\klam{B\ap,\rho\ap,\BBBB\ap}$ is a
complete LCA, $\p:\klam{B,\rho,\BBBB}\lra\klam{B\ap,\rho\ap,\BBBB\ap}$ is an LCA-embedding
and $\p(\BBBB)$ is
 dV-dense in $\klam{B\ap,\rho\ap,\BBBB\ap}$.

Two LCA-completions $(\p,\klam{B\ap,\rho\ap,\BBBB\ap})$ and
$(\psi,\klam{B'',\rho'',\BBBB''})$ of a local contact algebra
$\klam{B,\rho,\BBBB}$ are said to be {\em equivalent} if there exists
an LCA-isomorphism
$$\eta:\klam{B\ap,\rho\ap,\BBBB\ap}\lra\klam{B'',\rho'',\BBBB''}$$ such that
$\psi=\eta\circ\p$.

We define analogously the notions of {\em NCA-completion}\/ and
{\em equivalent NCA-comple\-tions}.
\end{defi}

Note that condition (LC3)
implies that if a set $D$ is  dV-dense in an LCA $\klam{B,\rho,\BBBB}$, then it is a
dense subset of $B$. Hence, if $(\p,\klam{B\ap,\rho\ap,\BBBB\ap})$ is
an LCA-completion of the LCA $\klam{B,\rho,\BBBB}$, then $(\p,B\ap)$ is
a   completion of the Boolean algebra $B$.

\begin{theorem}\label{lcacompletion}{\rm (\cite{D-AMH2-10,D-a0903-2593})}
Every local contact algebra $\klam{B,\rho,\BBBB}$ has a unique (up to
equivalence) LCA-completion $(\p,\klam{B\ap,\rho\ap,\BBBB\ap})$. If $X\df\LAM^a(\klam{B\ap,\rho\ap,\BBBB\ap})$, then
$\klam{B,\rho,\BBBB}$ has an LCA-completion of the form $(\psi,\LAM^t(X))$ which is equivalent to its LCA-completion $(\p,\klam{A\ap,\rho\ap,\BBBB\ap})$.

In particular, every normal contact algebra $\klam{B,C}$ has a unique (up to
equivalence) NCA-completion.
\end{theorem}

\section{Dimension of a  precontact algebra}

The following assertion might be known.

\begin{pro}\label{prodim}
Let $X$ be a normal $T_1$-space, and $n \in \NNN^-$. Then,
$\dim(X) \leq n$ \tiff for every finite regular open cover\/
$\UU \df  \set{U_1, \ldots, U_{n+2}}$ of $X$ there exists a regular
closed shrinking\/ $\FF \df  \set{F_1, \ldots, F_{n+2}}$ of\/ $\UU$
such that $\bigcap \FF = \ems$ (i.e., $\ord(\FF) \leq n$).
\end{pro}
\begin{proof}
($\Rightarrow$) Let $\dim(X)\le n$ and $\UU \df  \set{U_1, \ldots, U_{n+2}}$ be a regular open
cover of $X$. Then, by Theorem 1.6.10 of  \cite{eng_dim}, $\UU$
has an open shrinking $\WW\df \{W_1,\ldots,W_{n+2}\}$ such that
$\bigcap\WW=\ems$. Using Theorem 1.7.8 of \cite{eng_dim}, we find
a closed shrinking $\FF\ap\df \{F_1,\ldots,F_{n+2}\}$ of $\WW$.  Now,
Theorem 3.1.2 of \cite{eng_dim} gives us an open swelling
$\VV\df \{V_1,\ldots,V_{n+2}\}$ of $\FF\ap$ such that
 $\cl(V_i)\sbe W_i$,
for every $i=1,\ldots,n+2$. Set
$\FF \df \{\cl(V_1),\ldots,\cl(V_{n+2})\}$. Then $\FF$ is a regular
closed shrinking of $\UU$ and $\bigcap\FF=\ems$.

\noindent($\Leftarrow$) Let  $\UU\ap \df  \set{U_1, \ldots, U_{n+2}}$
be an  open cover of $X$. Then, by Theorem 1.7.8 of
\cite{eng_dim}, $\UU\ap$ has a closed shrinking
$\FF\ap\df \{F_1\ap,\ldots,F_{n+2}\ap\}$. Using Theorem 3.1.2 of
\cite{eng_dim}, we obtain an open swelling
$\VV\df \{V_1,\ldots,V_{n+2}\}$ of $\FF\ap$ such that
$\cl(V_i)\sbe U_i$, for every $i=1,\ldots,n+2$. Then
$\UU\df \{\int(\cl(V_1)),\ldots,\int(\cl(V_{n+2}))\}$ is a regular
open shrinking of $\UU\ap$. By our hypothesis, $\UU$ has a regular
closed shrinking $\FF \df  \set{F_1, \ldots, F_{n+2}}$  such that
$\bigcap \FF = \ems$. Then $\FF$ is a closed shrinking of
$\UU\ap$. By Theorem 3.1.2 of \cite{eng_dim}, $\FF$ has an open
swelling $\WW\df \{W_1,\ldots,W_{n+2}\}$ such that $\cl(W_i)\sbe
U_i$ for every $i=1,\ldots,n+2$; thus, $\WW$ is an open shrinking
of $\UU\ap$ and $\bigcap\WW=\ems$. Thus, by Theorem 1.6.10 of
\cite{eng_dim}, $\dim(X)\le n$.
\end{proof}


\begin{cor}\label{prodimc}
Let $X$ be a normal $T_1$-space, and $n \in \NNN^-$. Then,
$\dim(X) \leq n$ \tiff for every finite regular open cover\/
$\UU \df  \set{U_1, \ldots, U_{n+2}}$ of $X$ there exists a regular
closed shrinking\/ $\FF \df  \set{F_1, \ldots, F_{n+2}}$ of\/ $\UU$
such that $\bigcap \FF = \ems$  and $\bigcup_{i=1}^{n+2}\int(F_i)=X$.
\end{cor}

\begin{proof}
($\Rightarrow$) Repeat the proof of the $``$if" part of Proposition \ref{prodim} rewriting only the last sentence of it as follows:
Then $\FF$ is a regular
closed shrinking of $\UU$,  $\bigcap\FF=\ems$ and $\bigcup_{i=1}^{n+2}\int(\cl(V_i))=X$.

\noindent($\Leftarrow$) This follows from Proposition \ref{prodim}.
\end{proof}

Having in mind the proposition above, we introduce the
notion of {\em dimension of a  precontact algebra} $\klam{B,\rho}$, denoted by  $\dim_a (\klam{B,\rho})$.
For technical reasons, we even define a more general notion.

\begin{defi}\label{defidim}
 For a precontact algebra $\klam{B,\rho}$, $n\in\NNN^-$ and a subset $D$ of $B$ such that $0,1\in D$,
  set
 \begin{gather*}
 \dim_a (D;\klam{B,\rho})\le n,
\end{gather*}
if for all $a_1,\ldots,a_{n+2}, b_1,\ldots,b_{n+2}\in D$ such that $\bigvee_{i=1}^{n+2} b_i = 1$ and $b_i \ll a_i$ for all $i=1,\ldots,n+2$, there exist $c_1,\ldots,c_{n+2},d_1,\ldots,d_{n+2}\in D$  which satisfy the following conditions:

\begin{quote}
\begin{enumerate}
\renewcommand{\theenumi}{\ensuremath{(D\arabic{enumi})}}
\item $c_i \ll d_i\ll a_i$ for every $i=1,\ldots,n+2$. \label{D1'}\label{D1}
\item $\bigvee_{i = 1}^{n+2} c_i = 1$ and $\bigwedge_{i = 1}^{n+2} d_i = 0$. \label{D2'}\label{D2}
\end{enumerate}
\end{quote}

\noindent Furthermore, set $\dim_a (D;\klam{B,\rho}) \df -1$  \tiff
$|B|=1$ (i.e., $0=1$
in $B$).
Finally, for all $n\in\NNNN$, set
\begin{gather*}
\dim_a (D;\klam{B,\rho}) \df
\begin{cases}
  n, &\text{if }n - 1 < \dim_a (D;\klam{B,\rho}) \leq n, \\
  \infty, &\text{if $n < \dim_a (D;\klam{B,\rho})$ for all $n \in \NNNN^-$.}
\end{cases}
\end{gather*}

When $D=B$, we will write simply $\dim_a(\klam{B,\rho})$ instead of $\dim_a(B;\klam{B,\rho})$.
Also, if $\klam{B,\rho,\BBBB}$ is an LCA, then we replace $\klam{B,\rho}$ in above notation  with $\klam{B,\rho,\BBBB}$.
\end{defi}

\begin{theorem}\label{thdim}
Let $(X, \TT)$ be a normal  $T_1$-space and $n\in\NNN^-$. Then,
$\dim(X) \leq n$ \tiff $\dim_a(\klam{\RC(X), \rho_X}) \leq n$.
\end{theorem}
\begin{proof}

Set $B\df \RC(X)$.

\noindent($\Rightarrow$) Let $\dim(X) \leq n$. Let
$a_1,\ldots,a_{n+2}$, $b_1,\ldots,b_{n+2}\in B$, $b_i \ll a_i$ for
every $i=1,\ldots,n+2$, and $\bigvee_{i=1}^{n+2} b_i = 1$. Then
$b_i\sbe\int(a_i)$  for every $i=1,\ldots,n+2$. Since
$\bigcup_{i=1}^{n+2}b_i=X$, we obtain that
$\AA\df \{\int(a_i)\st i=1,\ldots,n+2\}$ is a regular open cover of
$X$. Then, by Corollary \ref{prodimc}, $\AA$ has a regular closed
shrinking $\DD\df \{d_1,\ldots,d_{n+2}\}$ such that
$\bigcap\DD=\ems$ and $\bigcup_{i=1}^{n+2}\int(d_i)=X$.
Now, using Proposition \ref{prodim}, we obtain a regular closed shrinking
$\CC\df \{c_1,\ldots,c_{n+2}\}$ of the regular open cover $\{\int(d_i)\st i=1,\ldots,n+2\}$ of $X$.
Then $c_i\ll d_i\ll a_i$ for every $i=1,\ldots,n+2$, $\bigvee_{i=1}^{n+2}c_i=1$ and $\bigwedge_{i=1}^{n+2}d_i=\cl(\int(\bigcap_{i=1}^{n+2}d_i))=0$.

\medskip

\noindent($\Leftarrow$) Let $\UU \df  \set{U_1, \ldots, U_{n+2}}$ be
a regular open cover of $X$. Then, by Theorem 1.7.8 of
\cite{eng_dim}, $\UU$ has a closed shrinking $\FF \df  \set{F_1,
\ldots, F_{n+2}}$. By Theorem 3.1.2 of \cite{eng_dim}, $\FF$ has
an open swelling $\VV\df \{V_1,\ldots,V_{n+2}\}$  such that
$\cl(V_i)\sbe U_i$, for every $i=1,\ldots,n+2$. Set $a_i\df \cl(U_i)$
and $b_i\df \cl(V_i)$, for every $i=1,\ldots,n+2$. Then
$b_i\sbe\int(a_i)$, i.e. $b_i\ll a_i$ for every $i=1,\ldots,n+2$.
Since $\bigcup\{\cl(V_i)\st i=1,\ldots,n+2\}=X$, we obtain that
 $\bigvee_{i=1}^{n+2} b_i = 1$. Thus, by our hypothesis, there exist
$c_1,\ldots,c_{n+2}, d_1,\ldots,d_{n+2}\in B$ such that $c_i\ll d_i\ll a_i$ for every
$i=1,\ldots,n+2$, $\bigvee_{i=1}^{n+2} c_i = 1$ and $\bigwedge_{i=1}^{n+2}d_i=0$.
Then $c_i\sbe\int(d_i)$, for every $i=1,\ldots,n+2$. Furthermore, we have that $\cl(\bigcap_{i=1}^{n+2}\int(d_i))=\cl(\int(\bigcap_{i=1}^{n+2}d_i))=\bigwedge_{i=1}^{n+2}d_i=\ems$.
Thus $\bigcap_{i=1}^{n+2}\int(d_i)=\ems$. Now we obtain that $\bigcap_{i=1}^{n+2}c_i\sbe\bigcap_{i=1}^{n+2}\int(d_i)=\ems$, and hence $\bigcap_{i=1}^{n+2}c_i=\ems$. Therefore, Proposition \ref{prodim} implies that $\dim(X)\le n$.
 \end{proof}

 \begin{cor}\label{thdimcor}
(a) If $\klam{B,\rho,\BBBB}$ is an LCA such that $\LAM^a(\klam{B,\rho,\BBBB})$ is a normal space, then $\dim_a(\klam{B,\rho,\BBBB})=\dim(\LAM^a(\klam{B,\rho,\BBBB})$.
In particular, for every NCA $\klam{B,\rho}$, we have that $\dim_a(\klam{B,\rho})=\dim(\LAM^a(\klam{B,\rho})$.

\smallskip

\noindent(b) If  $X$ is a normal locally compact $T_1$-space, then $\dim(X)=\dim_a(\LAM^t(X))$. In particular, for every compact Hausdorff space $X$, $\dim(X)=\dim_a(\LAM^t(X))$.
\end{cor}
 \begin{proof}
 This follows from Theorems \ref{thdim} and \ref{lccont}.
 \end{proof}

The next notion is analogous to the notions of $``$dense subset" and $``$dV-dense subset" regarded, respectively, in \cite{deV} and \cite{D-AMH2-10,D-a0903-2593}.

\begin{defi}\label{devdense}
Let $\klam{B,\rho}$ be a CA. A subset $D$  of $B$ is said to be {\em DV-dense in} $\klam{B,\rho}$ if it satisfies the following condition:


(DV) If $a,b\in B$  and  $a\ll b$  then there exists $c\in D$  such that  $a\ll c\ll b$.
\end{defi}

 \begin{lm}\label{lmdim1}
 Let $\klam{B,\rho}$ be a precontact algebra and $D$ be a DV-dense in $\klam{B,\rho}$.
 Then $\dim_a(\klam{B,\rho})=\dim_a(D;\klam{B,\rho})$.
 \end{lm}

 \begin{proof}
 If $\dim_a(D;\klam{B,\rho})=\infty$ then $\dim_a(\klam{B,\rho})\le\dim_a(D;\klam{B,\rho})$. Suppose that $\dim_a(D;\klam{B,\rho})=n$, where $n\in\NNNN^-$ and let
 $a_1,\ldots,a_{n+2}, b_1,\ldots,b_{n+2}\in B$ be such that $\bigvee_{i=1}^{n+2} b_i = 1$ and $b_i \ll a_i$ for all $i=1,\ldots,n+2$. Then, by (DV), there exist
 $c_1,\ldots,c_{n+2}, d_1,\ldots,d_{n+2}\in D$  such that  $b_i \ll c_i\ll d_i\ll a_i$ for all $i=1,\ldots,n+2$. Obviously, we have that $\bigvee_{i=1}^{n+2} c_i = 1$.
 Thus there exist $c_1\ap,\ldots,c_{n+2}\ap, d_1\ap,\ldots,d_{n+2}\ap\in D$ such that $c_i\ap\ll d_i\ap\ll d_i$ for all $i=1,\ldots,n+2$, $\bigvee_{i=1}^{n+2} c_i\ap = 1$ and $\bigwedge_{i=1}^{n+2} d_i\ap = 0$. Since $c_i\ap\ll d_i\ap\ll a_i$ for all $i=1,\ldots,n+2$ and $D\sbe B$, we obtain that $\dim_a(\klam{B,\rho})\le n$. So, we have proved that $\dim_a(\klam{B,\rho})\le\dim_a(D;\klam{B,\rho})$.

  For the other direction, let us prove that $\dim_a(\klam{B,\rho})\ge\dim_a(D;\klam{B,\rho})$. Obviously, if $\dim_a(\klam{B,\rho})=\infty$ then $\dim_a(D;\klam{B,\rho})\le\dim_a(\klam{B,\rho})$. Now, suppose that $\dim_a(\klam{B,\rho})=n$, where $n\in\NNNN^-$, and let
 $a_1,\ldots,a_{n+2}, b_1,\ldots,b_{n+2}\in D$ be such that $\bigvee_{i=1}^{n+2} b_i = 1$ and $b_i \ll a_i$ for all $i=1,\ldots,n+2$. Then
 there exist $c_1\ap,\ldots,c_{n+2}\ap, d_1\ap,\ldots,d_{n+2}\ap\in B$ such that $c_i\ap\ll d_i\ap\ll a_i$ for all $i=1,\ldots,n+2$, $\bigvee_{i=1}^{n+2} c_i\ap = 1$ and $\bigwedge_{i=1}^{n+2} d_i\ap = 0$. Now, by (DV),
 there exist
 $c_1,\ldots,c_{n+2}, d_1,\ldots,d_{n+2}\in D$  such that  $c_i\ap \ll c_i\ll d_i\ll d_i\ap$ for all $i=1,\ldots,n+2$. Obviously, we have  $\bigvee_{i=1}^{n+2} c_i = 1$ and $\bigwedge_{i=1}^{n+2} d_i = 0$.
   Since $c_i\ll d_i\ll a_i$ for all $i=1,\ldots,n+2$, we obtain  $\dim_a(D;\klam{B,\rho})\le n$. So, we have proved that $\dim_a(D;\klam{B,\rho})\le\dim_a(\klam{B,\rho})$, and therefore, $\dim_a(\klam{B,\rho})=\dim_a(D;\klam{B,\rho})$.
 \end{proof}

 \begin{theorem}\label{thcompldim}
 Let $\klam{B,C}$ be an NCA and $(\p,\klam{B\ap,C\ap})$ be the NCA-completion of it. Then $\dim_a(\klam{B,C})=\dim_a(\klam{B\ap,C\ap})$.
 \end{theorem}
 \begin{proof}
 By Definition \ref{defcompl} and Fact \ref{wfact}, $\p(B)$ is a DV-dense subset of $B\ap$. Thus, by Lemma \ref{lmdim1}, $\dim_a(\klam{B\ap,C\ap})=\dim_a(\p(B);\klam{B\ap,C\ap})$.
 Hence, our assertion follows from the fact that $\dim_a(\klam{B,C})=\dim_a(\p(B);\klam{B\ap,C\ap})$.
 \end{proof}

 \begin{pro}\label{badim}
 Let $B$ be a non-degenerate Boolean algebra (i.e., $|B|>1$). Then $\dim_a(\klam{B,\rho_s})=0=\dim_a(\klam{B,\rho_l})$ (see Example \ref{extrcr} for $\rho_s$ and $\rho_l$).
 \end{pro}
 \begin{proof}
 Since $|B|>1$, we have  $\dim_a(\klam{B,\rho_s})>-1$ and $\dim_a(\klam{B,\rho_l})>-1$. So, we need to show that $\dim_a(\klam{B,\rho_s})\le 0$ and $\dim_a(\klam{B,\rho_l})\le 0$.

 We will first prove that $\dim_a(\klam{B,\rho_s})\le 0$. Recall that in $\klam{B,\rho_s}$, $a\ll b$ \tiff $a\le b$.
 So, let $a_1,a_2,b_1,b_2\in B$, $b_1\vee b_2=1$ and $b_i\le a_i$ for $i=1,2$. Then $a_1\vee a_2 = 1$. Set $a\df a_1\we a_2$, $c_1=d_1\df a^*\we a_1$ and $c_2=d_2\df a_2$.
 Then $c_1\le d_1\le a_1$, $c_2\le d_2\le a_2$, $c_1\vee c_2=(a^*\we a_1)\vee a_2= ((a_1^*\vee a_2^*)\we a_1)\vee a_2=(a_2^*\we a_1)\vee a_2= (a_1\vee a_2)\we(a_2\vee a_2^*)=1$ and $d_1\we d_2=(a^*\we a_1)\we a_2=(a_1\we a_2)^*\we(a_1\we a_2)=0$. Thus, $\dim_a(\klam{B,\rho_s})\le 0$, and altogether $\dim_a(\klam{B,\rho_s})=0$.

  Next, we will  prove that $\dim_a(\klam{B,\rho_l})\le 0$. It is easy to see that in $\klam{B,\rho_l}$, $a\ll b$ \tiff $a=0$ or $b=1$.
 So, let $a_1,a_2,b_1,b_2\in B$, $b_1\vee b_2=1$ and $b_i\ll a_i$ for $i=1,2$. Then $b_i=0$ or $a_i=1$, for $i=1,2$. We will consider all possible cases.

 {\em Case 1.} Let $b_1=0$. Then $b_2=1$ and hence $a_2=1$. However, $a_1$ could be equal to $0$ or to $1$. In both cases,
   setting $c_1=d_1\df 0$ and $c_2=d_2\df 1$, we obtain  $\dim_a(\klam{B,\rho_l})\le 0$.

{\em Case 2.} Let $b_2=0$. Then we argue analogously (just interchange the indices).

{\em Case 3.} Let $a_1=1$. Since $a_1\vee a_2=1$, $a_2$ could be equal to $0$ or to $1$.

{\em Case 3a.} Let $a_2=0$. Setting $c_1=d_1\df 1$ and $c_2=d_2\df 0$, we obtain  $\dim_a(\klam{B,\rho_l})\le 0$.

{\em Case 3b.} Let $a_2=1$. Setting $c_1=d_1\df 0$ and $c_2=d_2\df 1$, we obtain  $\dim_a(\klam{B,\rho_l})\le 0$.

 {\em Case 4.} Let $a_2=1$. Then we argue analogously to Case 3 (just interchange the indices).

 Thus, we have shown  that $\dim_a(\klam{B,\rho_l})=0$.
 \end{proof}

It is well known that for a normal $T_1$-space $X$ and a regular closed subset $M$ of $X$, $\dim(M) \leq \dim(X)$ holds (this is true even for closed subsets $M$ of $X$, see e.g.  \cite{eng_dim}). According to Theorems \ref{lccont} and \ref{thdim}, the dual of this assertion is the following one:  if $X$ is a normal  locally compact $T_1$-space  and $M\in RC(X)$, then $\dim_a(\LAM^t(M))\le \dim_a(\LAM^t(X))$. Theorem \ref{conregclo} describes the LCA $\LAM^t(M)$ in terms of the LCA $\LAM^t(X)$, so that we can reformulate the above statement in a purely algebraic terms. We will supply this new statement with an algebraic proof, obtaining in this way an algebraic generalization of the topological statement stated above.  (Note that we will just take an LCA without requiring that it is dual to a {\em normal}\/  locally compact $T_1$-space.)

\begin{pro}\label{pro:dimrel}
Suppose that $\klam{B, \rho, \BBBB}$ is an LCA, $m \in B^+$, and $\klam{B_m, \rho_m, \BBBB_m}$ is the relative LCA of $\klam{B, \rho, \BBBB}$, i.e.

\begin{gather*}
\rho_m \df \rho \restrict B_m^2, \ \BBBB_m \df \set{b \we m: b \in \BBBB}.
\end{gather*}

Then, $\dim_a(\klam{B_m, \rho_m, \BBBB_m}) \leq \dim_a(\klam{B, \rho, \BBBB})$.
\end{pro}
\begin{proof}
Recall that $B_m\df \{b\in B\st b\le m\}$. We denote the complement in $B_m$ by ${}^{*_m}$, i.e. $a^{*_m} \df a^* \land m$. Note that, for $a,b\in B_m$, $a\ll_m b$ means that $a(-\rho)b^{*_m}$, i.e.,
$a(-\rho)(b^*\we m)$. Clearly, if $a,b\in B_m$ and  $a\ll_m b$, then $b^{*_m}\ll_m a^{*_m}$.

Let $\dim_a(\klam{B, \rho, \BBBB}) = n$, and suppose that $a_1, \ldots, a_{n+2}, b_1, \ldots, b_{n+2} \in B_m$ are such that $\bigvee_{i=1}^{n+2} b_i = m$ and $b_i \ll_m a_i$.
Then $a_i^* \land m \ll m$. Indeed, we have that $b_i \ll_m a_i$ for $i=1,\ldots, n+2$. Thus $a_i^{*_m} \ll_m b_i^{*_m}$, i.e. $(a_i^*\we m)\ll (b_i^*\we m)$ for $i=1,\ldots, n+2$. Since $b_i^*\we m\le m$ for $i=1,\ldots, n+2$, \ref{di3} implies that $a_i^*\we m\ll m$ for $i=1,\ldots, n+2$.

Now, set $$a_i' \df a_i \lor m^* \mbox{ and }b_i' \df b_i \lor m^*$$ for all $1 \leq i \leq n+2$; clearly, $\sum_{i=1}^{n+2} b_i' = 1$. Furthermore, $b_i' \ll a_i'$. Indeed,  assume not; then $b_i' \rho {(a_i')}^*$, i.e. $(b_i \lor m^*) \rho (a_i^* \land m)$. If $b_i \rho (a_i^* \land m)$, then $b_i \not\ll_m a_i$, and if $m^* \rho (a_i^* \land m)$, then $a_i^* \land m \not\ll m$, a contradiction in both cases.

 Since $\dim_a(\klam{B, \rho, \BBBB}) = n$, there exist $c_1, \ldots, c_{n+2}, d_1 \ldots, d_{n+2} \in B$ such that $$\bigvee_{i = 1}^{n+2} c_i = 1,\ \   \bigwedge_{i = 1}^{n+2} d_i = 0, \mbox{ and }c_i \ll d_i\ll a_i'$$ for every $i=1,\ldots,n+2$.  Set $$s_i \df c_i \land m \ \ \mbox{ and }\ \ t_i \df d_i \land m.$$ Clearly, $\bigvee_{i = 1}^{n+2} s_i = m$ and $\bigwedge_{i = 1}^{n+2} t_i = 0$. All that is left to show is $s_i \ll_m t_i \ll_m a_i$. We have that
\begin{xalignat*}{2}
s_i \ll_m t_i &\Iff s_i(-\rho)t_i^{*_m}
\\
&\Iff s_i(-\rho) (t_i^* \land m), \\
&\Iff (c_i \land m)(-\rho)((d_i \land m)^* \land m) \\
&\Iff (c_i \land m)(-\rho)(d_i^* \land m).
\end{xalignat*}
Now, $(c_i \land m)(-\rho)(d_i^* \land m)$ is implied by $c_i \ll d_i$.

Similarly,
\begin{xalignat*}{2}
t_i \ll_m a_i &\Iff t_i(-\rho) (a_i^* \land m), \\
&\Iff (d_i \land m)(-\rho)(a_i^* \land m).
\end{xalignat*}
Since $d_i \ll a_i'$, i.e. $d_i(-\rho)(a_i \lor m^*)^*$, we see that $(d_i \land m)\rho(a_i^* \land m)$ is impossible, and it follows that $t_i \ll_m a_i$.
\end{proof}

\section{Weight of a  local contact algebra}

 In this section, we are going to define the notions of base and weight of an LCA $\underline{B}\df  \klam{B, \rho, \BBBB}$ in such a way that if $\underline{B}$ is complete, then the weight of $\underline{B}$ is equal to the weight of the space $\LAM^a(\underline{B})$, equivalently, if $X$ is a locally compact Hausdorff space, then the weight of $X$ is equal to the weight of $\LAM^t(X)$. Clearly, the main step is to define an adequate notion of base for a complete LCA
 $\underline{B}$.  In doing this, we
 use the fact that the family $\RO(X)=\{\int(F)\st F\in\RC(X)\}$ is an open base for $X$ (because $X$ is regular) and hence, by the
Alexandroff-Urysohn theorem \cite[Theorem 1.1.15]{E}, $\RO(X)$ has
a subfamily $\BB$, with $|\BB|=w(X)$, which is a
base for $X$.

The next definition and theorem generalize the analogous
definition and theorem of de Vries \cite{deV}. Note that our
$``$\/\/base" (see the definition below) appears in \cite{deV} (for
NCAs) as $``$dense set".

\begin{defi}\label{clcawe}
\rm Let $\klam{B,\rho,\BBBB}$ be an LCA and $D$ be a subset of
$\BBBB$. Then $D$ is called a {\em base}
for $\klam{B,\rho,\BBBB}$ if it is  dV-dense in $\klam{B,\rho,\BBBB}$.
The cardinal
number
$$w_a(\klam{B,\rho,\BBBB})\df \min\{\card{D} \st D \mbox{ is a base for } \klam{B,\rho,\BBBB}\}$$
is called the {\em weight of} $\klam{B,\rho,\BBBB}$.
\end{defi}

\begin{lm}\label{weightpiwlm1}
Let $X\in|\HLC|$ and $\DD$ be a base for the LCA $\LAM^t(X)$. Then $$\BB_{\DD}\df\{\int(F)\st F\in\DD\}$$ is a base for $X$.
\end{lm}
\begin{proof}
Let $x\in X$ and $U$ be a neighborhood of $x$. Since $X$ is regular and locally compact, there exist $F,G\in \CR(X)$ such that $x\in\int(F)\sbe F\sbe\int(G)\sbe G\sbe U$.  Then  $F\ll_{\rho_X} G$. Hence, there exists $H\in \DD$ such that $F\sbe H\sbe G$. It follows that  $\int(H)\in\BB_{\DD}$ and $x\in\int(H)\sbe U$. So, $\BB_{\DD}$ is a base for $X$.
\end{proof}

\begin{lm}\label{weightpiwlm2}
Let $X\in|\HLC|$, $\BB$ be a base for $X$ and $Cl(\BB)\df\{\cl(U)\st U\in\BB\}\sbe \CR(X)$.  Then, the sub-join-semilattice $\LL_J(\BB)$ of $\CR(X)$ generated by  $Cl(\BB)$  is a base for the LCA $\LAM^t(X)$.
\end{lm}
\begin{proof}
Let $F,G\in \CR(X)$ and $F\ll_{\rho_X} G$, i.e. $F\sbe \int(G)$. By regularity, for every $x\in F$ there exists $U_x\in \BB$
such that $x\in U_x\sbe \cl(U_x)\sbe \int(G)$. Since $F$ is compact, there exist $n\in\NNNN^+$ and $x_1,\ldots,x_n\in F$ such that $F\sbe\bigcup_{i=1}^n U_{x_i}\sbe\bigcup_{i=1}^n \cl(U_{x_i})\sbe\int(G)$. Thus
$H\df\bigcup_{i=1}^n \cl(U_{x_i})=\bigvee_{i=1}^n \cl(U_{x_i})\in\LL_J(\BB)$ and $F\sbe H\sbe G$. So, $\LL_J(\BB)$ is a base for the LCA $\klam{\RC(X),\rho_X,\CR(X)}$.
\end{proof}

\begin{theorem}\label{weightpiw}
Let  $X$ be a locally compact Hausdorff space and $w(X)\ge\aleph_0$.
Then $w(X)=w_a(\klam{\RC(X),\rho_X,\CR(X)})$ (i.e., $w(X)=w_a(\LAM^t(X))$).
\end{theorem}
\begin{proof}
We know that the family $\BB_0\df \{\int(F)\st
F\in \CR(X)\}$ is a base for $X$. Hence, by the Alexandroff-Urysohn theorem for bases \cite[Theorem 1.1.15]{E},
there exists a base $\BB$ of $X$ such that
$\BB\sbe\BB_0$ and $\card{\BB}=w(X)$. Let $\LL_J(\BB)$ be the
sub-join-semilattice of $\CR(X)$  generated  by the set
$\{\cl(U)\st U\in\BB\}$. Then, by Lemma \ref{weightpiwlm2}, $\LL_J(\BB)$
is a base for $\klam{\RC(X),\rho_X,\CR(X)}$.  Clearly, $|\LL_J(\BB)|=|\BB|=w(X)$.
Hence, $w(X)\ge w_a(\klam{\RC(X),\rho_X,\CR(X)})$.

Conversely, let $\DD$ be a base for $\klam{\RC(X),\rho_X,\CR(X)}$ such that
$$\card{\DD}=w_a(\klam{\RC(X),\rho_X,\CR(X)}).$$ Then, by Lemma \ref{weightpiwlm1},
$\BB_{\DD}\df\{\int(F)\st F\in\DD\}$ is a base for $X$.
Since $|\BB_{\DD}|=|\DD|$, we obtain that
$w(X)\le w_a(\klam{\RC(X),\rho_X,\CR(X)})$.

Altogether, we have shown  that $w(X)=w_a(\klam{\RC(X),\rho_X,\CR(X)})$.
\end{proof}

\begin{lm}\label{complwelm}
Let $\klam{B,\rho,\BBBB}$ be an LCA and $(\p,\klam{B\ap,\rho\ap,\BBBB\ap})$ be its LCA-completion. Then:

\smallskip

\noindent(a) if $D$ is a base for $\klam{B,\rho,\BBBB}$, then $\p(D)$ is a base for $\klam{B\ap,\rho\ap,\BBBB\ap}$;

\smallskip

\noindent(b) if $D\ap$ is a base for  $\klam{B\ap,\rho\ap,\BBBB\ap}$ and $D\ap\sbe\p(\BBBB)$, then $\p\inv(D\ap)$ is a base for $\klam{B,\rho,\BBBB}$.
\end{lm}
\begin{proof} By definition,  $\p(\BBBB)$ is  dV-dense in $\klam{A\ap,\rho\ap,\BBBB\ap}$.

\noindent(a) Let $a,c\in\BBBB\ap$ and $a\ll_{\rho\ap} c$. Then, by Fact \ref{wfact}, there exist $b_1,b_2\in\BBBB$ such that $a\ll_{\rho\ap} \p(b_1)\ll_{\rho\ap} \p(b_2)\ll_{\rho\ap} c$; thus $b_1\ll_\rho b_2$. Hence, there exists some $b\in D$ such that $b_1\ll_\rho b\ll_\rho b_2$. Then, $a\ll_{\rho\ap} \p(b)\ll_{\rho\ap}  c$, and therefore,  $\p(D)$ is a base for $\klam{A\ap,\rho\ap,\BBBB\ap}$.

\smallskip

\noindent(b) This is obvious.
\end{proof}

\begin{theorem}\label{complwe}
 Let $\klam{B,\rho,\BBBB}$ be an LCA,  $(\p,\klam{B\ap,\rho\ap,\BBBB\ap})$ be its LCA-completion  and $w_a(\klam{B,\rho,\BBBB})\ge\aleph_0$. Then $w_a(\klam{B,\rho,\BBBB})=w_a(\klam{B\ap,\rho\ap,\BBBB\ap})$.
 \end{theorem}
\begin{proof}
%
Let  $X\df \LAM^a(\klam{B\ap,\rho\ap,\BBBB\ap})$.
Then, by Theorem \ref{lcacompletion}, we may suppose \wlg that $\klam{B,\rho,\BBBB}$ is an LC-subalgebra of $\LAM^t(X)=\klam{\RC(X),\rho_X,\CR(X)}$ and $(id,\LAM^t(X))$ is an
LCA-completion of $\klam{B,\rho,\BBBB}$, where $id:\klam{B,\rho,\BBBB}\lra \LAM^t(X)$ is the inclusion map; also, $(id,\LAM^t(X))$ and $(\p,\klam{B\ap,\rho\ap,\BBBB\ap})$ are equivalent LCA-completions of $\klam{B,\rho,\BBBB}$ (recall also that, by Theorem \ref{lccont}, $\LAM^t(X)$ and $\klam{B\ap,\rho\ap,\BBBB\ap}$ are LCA-isomor\-phic). So, $\BBBB$ is  dV-dense in $\LAM^t(X)$. Thus  $\BBBB$ is a base for $\LAM^t(X)$. Let $\DD$ be a base for $\klam{B,\rho,\BBBB}$ and $|\DD|=w_a(\klam{B,\rho,\BBBB})$. Then, by Lemma \ref{complwelm}(a),  $\DD$ is a base for $\LAM^t(X)$. Therefore,
$w_a(\klam{B\ap,\rho\ap,\BBBB\ap})\le |\DD|=w_a(\klam{B,\rho,\BBBB})$.
Further, by Lemma
\ref{weightpiwlm1},  $\BB_{\DD}\df\{\int(F)\st F\in \DD\}$ is a base for $X$.
Applying the Alexandroff-Urysohn theorem for bases
\cite[Theorem 1.1.15]{E}, we find a base $\BB$ for $X$ such that $\BB\sbe\BB_{\DD}$ and $|\BB|=w(X)$.
Then, Lemma \ref{weightpiwlm2} implies that the
sub-join-semilattice $\LL_J(\BB)$ of $\CR(X)$, generated by the set $Cl(\BB)\df\{\cl(U)\st U\in \BB\}$, is a base for $\LAM^t(X)$. Since $\BB\sbe\BB_{\DD}$, we have  $Cl(\BB)\sbe \DD$.
On the other hand, $\DD\sbe \BBBB$ and $\BBBB$ is a sub-join-semilattice of $\CR(X)$; hence $\LL_J(\BB)\sbe \BBBB$. Then, by Lemma \ref{complwelm}(b),  $\LL_J(\BB)$ is a base for $\klam{B,\rho,\BBBB}$. Thus, using Theorem \ref{weightpiw}, we obtain  $$w_a(\klam{B,\rho,\BBBB})\le |\LL_J(\BB)|=|\BB|=w(X)=w_a(\klam{B\ap,\rho\ap,\BBBB\ap})\le w_a(\klam{B,\rho,\BBBB}).$$
 So,
$w_a(\klam{B,\rho,\BBBB})= w_a(\klam{B\ap,\rho\ap,\BBBB\ap})$.
\end{proof}

The next theorem is an analogue of the Alexandroff-Urysohn theorem for bases \cite[Theorem 1.1.15]{E}.

\begin{theorem}\label{auth}
Let $D$ be a base for an  LCA $\klam{B,\rho,\BBBB}$ with infinite weight. Then there exists a subset $D_1$ of $D$ such that $|D_1|=w_a(\klam{B,\rho,\BBBB})$ and the sub-join-semilattice $L$ of $\BBBB$, generated by $D_1$, is a base for $\klam{B,\rho,\BBBB}$ with cardinality $w_a(\klam{B,\rho,\BBBB})$. If $D$ is, in addition, a
sub-join-semilattice  of $\BBBB$, then $L\sbe D$.
\end{theorem}
\begin{proof}
Let $(\p,\klam{B\ap,\rho\ap,\BBBB\ap})$ be the LCA-completion of $\klam{B,\rho,\BBBB}$.   As in the proof of Theorem \ref{complwe}, we set $X\df\LAM^a(\klam{B\ap,\rho\ap,\BBBB\ap})$ and   suppose \wlg that $\klam{B,\rho,\BBBB}$ is an LC-subalgebra of $\LAM^t(X)$. Then, by Lemma \ref{complwelm}(a), $D$ is a base for $\LAM^t(X)$. Thus, by Lemma \ref{weightpiwlm1}, $\BB_D\df\{\int(F)\st F\in D\}$ is a base for $X$. Using \cite[Theorem 1.1.15]{E}, we obtain a base $\BB$ for $X$ such that $\BB\sbe\BB_D$ and $|\BB|=w(X)$. Let $D_1\df \{\cl(U)\st U\in\BB\}$. Then $D_1\sbe D\sbe\BBBB$ and, by Lemma \ref{weightpiwlm2},  the sub-join-semilattice $L$ of $\CR(X)$, generated by $D_1$, is a base for $\LAM^t(X)$. Since $L\sbe \BBBB$,  Lemma \ref{complwelm}(b) implies that $L$ is a base for $\klam{B,\rho,\BBBB}$. Clearly,  $L$ coincides with the sub-join-semilattice  of $\BBBB$, generated by $D_1$. Using Theorems \ref{weightpiw} and
\ref{complwe}, we obtain  $$|L|=|D_1|=|\BB|=w(X)=w_a(\klam{B\ap,\rho\ap,\BBBB\ap})=w_a(\klam{B,\rho,\BBBB}).$$
\end{proof}

\begin{pro}\label{infwe}
If $\klam{B,\rho,\BBBB}$  is an LCA and $|B|\ge\aleph_0$ then $w_a(\klam{B,\rho,\BBBB})\ge\aleph_0$.
\end{pro}
\begin{proof}
Let $(\p,\klam{B\ap,\rho\ap,\BBBB\ap})$ be the LCA-completion of $\klam{B,\rho,\BBBB}$.   As in the proof of Theorem \ref{complwe}, we set $X\df\LAM^a(\klam{B\ap,\rho\ap,\BBBB\ap})$ and  suppose \wlg that $\klam{B,\rho,\BBBB}$ is an LC-subalgebra of $\LAM^t(X)$. Then $B\sbe \RC(X)$, and thus $|\RC(X)|\ge\aleph_0$. Assume  that $w(X)$ is finite. Then $X$ is a discrete space and $w(X)=|X|$. Thus $RC(X)$ is finite, a contradiction. Therefore, $w(X)\ge\aleph_0$.
From Theorems \ref{weightpiw} and
\ref{complwe}, we obtain  $$w_a(\klam{B,\rho,\BBBB})=w_a(\klam{B\ap,\rho\ap,\BBBB\ap})=w_a(\LAM^t(X))=w(X)\ge\aleph_0.$$
\end{proof}

\begin{theorem}\label{metrclca}
Let $X\in\card\HLC$. Then $X$ is metrizable iff there exists a
set\/ $\GA$ and a family $\{\klam{B_\g,\rho_\g,\BBBB_\g}\st\g\in\GA\}$
of complete LCAs such that
$$\LAM^t(X)=\prod\{\klam{B_\g,\rho_\g,\BBBB_\g}\st\g\in\GA\}$$ and, for
each $\g\in\GA$, $w_a(\klam{B_\g,\rho_\g,\BBBB_\g})\le\aleph_0$.
\end{theorem}
\begin{proof} It is well known that
  a locally
compact Hausdorff space is metrizable \tiff it is a topological sum
of locally compact Hausdorff spaces with countable weight (see, e.g., \cite[p. 315]{Alex} or  \cite[Theorem 5.1.27]{E}). Since
 $\LAM^t$ is a duality functor, it
converts the $\HLC$-sums in $\DLC$-products. Hence, our assertion
follows from the  theorem cited above and Theorems
\ref{prodclcacat} and \ref{weightpiw}.
\end{proof}

\begin{cor}\label{cordim}
If $\klam{B, \rho, \BBBB}$ is a complete LCA and $w_a(\klam{B, \rho,
\BBBB}) \leq \aleph_0$, then $\LAM^a(\klam{B, \rho, \BBBB})$ is a metrizable, separable, locally compact space.
\end{cor}

\begin{nota}\label{as}
\rm  Let $\klam{A,\rho,\BBBB}$ be an LCA. We set
$$\klam{A,\rho,\BBBB}_S\df \{a\in A\st a\ll_\rho a\}.$$ We will  write
simply $``A_S$" instead of $``\klam{A,\rho,\BBBB}_S$" when this does
not lead to an ambiguity.
\end{nota}

\begin{theorem}\label{nuldim}
Let  $\klam{B,\rho,\BBBB}$ be an LCA and $(\p,\klam{B\ap,\rho\ap,\BBBB\ap})$ be its LCA-comple\-tion. Then the space
$\LAM^a(\klam{B\ap,\rho\ap,\BBBB\ap})$ is zero-dimensional \tiff the set
$B_S\cap\BBBB$ is a base for $\klam{B,\rho,\BBBB}$.
\end{theorem}

\begin{proof}  Set $X\df\LAM^a(\klam{B\ap,\rho\ap,\BBBB\ap})$. As in the proof of Theorem \ref{complwe}, we may suppose \wlg that $\klam{B,\rho,\BBBB}$ is an LC-subalgebra of $\LAM^t(X)$, and
that $\BBBB$ is  dV-dense in $\LAM^t(X)$.
Then,  it follows from Lemma \ref{weightpiwlm1} that the set $\BB_{\BBBB}\df\{\int(F)\st F\in\BBBB\}$ is a base for $X$.

\noindent($\Rightarrow$)  Let $X$ be zero-dimensional. Then there exists a base $\BB$ for $X$ consisting of clopen compact sets.  Clearly, for every $U\in\BB$, we have  $U\ll_{\rho_X} U$. Since $\BBBB$ is  dV-dense in $\LAM^t(X)$, we obtain  $\BB\sbe B_S\cap\BBBB$. Therefore, $B_S\cap\BBBB$ is a base for $X$. Since $B_S\cap\BBBB$ is closed under joins, Lemma \ref{weightpiwlm2} implies that $B_S\cap\BBBB$ is a base for $\LAM^t(X)$. Then, using Lemma \ref{complwelm}(b), we obtain that $B_S\cap\BBBB$ is a base for $\klam{B,\rho,\BBBB}$.

\noindent($\Leftarrow$) Let $x\in X$ and $U$ be a neighborhood of
$x$. Since $\BB_{\BBBB}$ is a base for $X$, there exist $a,b\in\BBBB$ such that
$x\in\int(a)\sbe a \sbe\int(b)\sbe b\sbe U$; hence, $a\ll_\rho
b$. Thus, there exists some $c\in B_S\cap\BBBB$ such that $a\le c\le
b$. Since $c$ is clopen in $X$ and $x\in c\sbe U$, it follows that $X$ has a base consisting of clopen sets, i.e. $X$ is zero-dimensional.
\end{proof}

In the sequel, we will denote by $K$  the Cantor set.

 Note that
$\RC(K)$ is isomorphic to the  completion $A$ of a free
Boolean algebra $A_0$ with $\aleph_0$ generators, Equivalently,
$\RC(K)$ is the unique (up to isomorphism)  atomless complete
Boolean algebra $A$  containing a countable dense subalgebra $A_0$
(see, e.g., \cite[Example 7.24]{kop89}).   Defining in $A$ a relation $\rho$ by
$a(-\rho)b$  \tiff there exists  some $c\in A_0$ such
that $a\le c\le b^*$, we obtain (as we will see below) that $\klam{A,\rho}$ is a complete NCA which is
NCA-isomorphic to the complete NCA $\klam{\RC(K),\rho_K}$. We will now present a generalization
of this construction.

We  denote by $\Bool$  the category of all Boolean algebras and Boolean homomorphisms, by $\Stone$
 the category of all compact zero-dimensional Hausdorff spaces and continuous maps, and by
$S^a:\Bool\lra\Stone$  the Stone duality functor (see, e.g., \cite{kop89}).

\begin{theorem}\label{nuldimlm}
Let $A_0$ be a dense  Boolean subalgebra of a Boolean algebra $A$.
 For all $a,b\in A$, set  $a\ll_\rho b$ if there exists some
 $c\in A_0$ such that $a\le c\le b$.
Then the following holds:

\smallskip

\noindent(a)   $\klam{A,\rho}$ is an NCA,
$\klam{A,\rho}_S=A_0$, $A_0$ is the smallest base for $\klam{A,\rho}$ and
$$w(\klam{A,\rho})=\card{A_0}.$$

\smallskip

\noindent(b) If $A$ is complete, then $\LAM^a(\klam{A,\rho})$ is homeomorphic to $S^a(A_0)$, and
  $(i_0,\klam{A,\rho})$ is an NCA-completion of the NCA
$\klam{A_0,\rho_s^{A_0}}$, where $i_0:A_0\lra A$ is the inclusion map.
\end{theorem}

\begin{proof} (a) It is easy to check that the
relation $\rho$ satisfies conditions ($\ll 1$)-($\ll 7$). To
establish ($\ll 5$) and ($\ll 6$), use the fact that for every
$c\in A_0$ we have, by the definition of the relation $\llx$, that
$c\llx c$. Hence, $\klam{A,\rho}$ is an NCA. By  definition of the relation $\ll_\rho$, we obtain  for $c\in A$, $c\ll_\rho c$ \tiff $c\in A_0$;
thus, $\klam{A,\rho}_S=A_0$. Obviously, $A_0$ is the smallest base for $\klam{A,\rho}$; hence, $w(\klam{A,\rho})=\card{A_0}$.

\smallskip

\noindent(b) Let $A$ be complete and set $X\df S^a(A_0)$. Then, the Stone map $s:A_0\lra \CO(X)$ is a Boolean isomorphism. Let $i:\CO(X)\lra \RC(X)$ be the inclusion map.
Then $(i\circ s, \RC(X))$ is a completion of $A_0$. We know that $(i_0,A)$ is a completion of $A_0$. Thus, there exists a Boolean isomorphism $\p:A\lra \RC(X)$ such that $\p\circ i_0=i\circ s$. We will show that $\p:\klam{A,\rho}\lra \klam{\RC(X),\rho_X}$
is an NCA-isomorphism. Let $a,b\in A$ and $a\ll_\rho b$. Then, there exists some $c\in A_0$ such that $a\le c\le b$. Thus, $\p(a)\le\p(c)\le\p(b)$. We have  $\p(A_0)=CO(X)$; hence, $\p(c)\in \CO(X)$. Therefore, $\p(a)\sbe\int(\p(b))$, i.e. $\p(a)\ll_{\rho_X}\p(b)$. Conversely, let $F,G\in \RC(X)$ and $F\ll_{\rho_X} G$, i.e. $F\sbe \int(G)$. Since $\CO(X)$ is a base of $X$, $F$ is compact and $\CO(X)$ is closed under finite unions, we obtain that there exists some $U\in \CO(X)$ such that $F\sbe U\sbe\int(G)\sbe G$. Then, $\p\inv(U)\in A_0$ and $\p\inv(F)\le\p\inv(U)\le\p\inv(G)$. Thus, by the definition of $\rho$, we obtain $\p\inv(F)\ll_\rho\p\inv(G)$.
Therefore, $\p:\klam{A,\rho}\lra \klam{\RC(X),\rho_X}$ is an NCA-isomorphism. Since $\klam{\RC(X),\rho_X}=\LAM^t(X)$ and $\LAM^a(\p):\LAM^a(\LAM^t(X))\lra\LAM^a(\klam{A,\rho})$ is a homeomorphism, we obtain that $\LAM^a(\klam{A,\rho})$ is homeomorphic to $S^a(A_0)$, using Theorem \ref{lccont}.

As we have seen in (a), $A_0$ is a base for $\klam{A,\rho}$, and thus, $A_0$ is  dV-dense in $\klam{A,\rho}$. Hence, for proving that $(i_0,\klam{A,\rho})$
is an NCA-completion of
$\klam{A_0,\rho_s^{A_0}}$, we need only  show that $\rho\cap (A_0\times A_0)=\rho_s^{A_0}$. So, let $a,b\in A_0$. Then, $$a(-\rho)b\Iff(\ex c\in A_0)(a\le c\le b^*).$$ Clearly, $a(-\rho)b$ implies that $a\we b=0$, i.e., $a(-\rho_s^{A_0})b$. Conversely, if  $a(-\rho_s^{A_0})b$, then  $a\we b=0$; hence, $a\le b^*$. Since $a\le a\le b^*$ and $a\in A_0$, we obtain that $a(-\rho)b$. Therefore, for every $a,b\in A_0$, we have  $a\rho_s^{A_0}b$ \tiff $a\rho b$.
\end{proof}

\section{Algebraic density and weight}

One may wonder why we do not define the notion of weight of a local contact algebra, or, more generally, of a Boolean algebra, in a much simpler way, based on the following reasoning: if $X$ is a semiregular space, then $\RO(X)$ is a base for $X$; thus, by \cite[Theorem 1.1.15]{E}, $\RO(X)$ contains a subfamily $\BB$  such that $\BB$ is a base for $X$ and $|\BB|=w(X)$; clearly, if $X$ is semiregular, then a subfamily $\BB$ of $\RO(X)$ is a base for $X$ \tiff for any $U\in \RO(X)$, we have $U=\bigcup\{V\in\BB\st V\sbe U\}$.

Having this in mind, it would be natural to define the weight of a Boolean algebra $B$ as the smallest cardinality of subsets $M$ of $B$ such that
for each  $b \in B$, $$b = \bigvee\set{x \in M\st x \leq b}.$$
The obtained cardinal invariant is well known in the theory of Boolean algebras as the  \emph{density} or {\em $\pi$-weight} (and even {\em pseudoweight})  of $B$ and is denoted by $\pi w(B)$ (see, e.g., \cite{kop89,monk14,Douwen}), but we will denote it by $\pi w_a(B)$. So,
$$\pi w_a(B)\df\min\{|M|\st (\fa b\in B)(b = \bigvee\set{x \in M\st x \leq b})\}.$$
 It is easy to see that $\pi w_a(B)$ is equal to the smallest cardinality of a dense subset of $B$ (see \cite[Lemma 4.9.]{kop89}).  Clearly, if $B$ is a dense subalgebra of $A$, then $\pi w_a(B) = \pi w_a(A)$; in particular, $B$ has the same density as its completion. Observe that a Boolean algebra has infinite $\pi$-weight \tiff it is infinite.

However, owing to the fact that in $\RO(X)$ the union is not equal to the join, $\pi w_a(\RO(X))$ may be strictly smaller than the weight of a space $X$, even when $X$ is  semiregular. It is well known that $\pi w_a$ corresponds to the topological notion of $\pi$-weight. Recall that a \emph{$\pi$--base} for a topological space $(X,\TT)$ is a subfamily $\PP$ of $\TT\stm\{\ems\}$ such that for every $U\in\TT\stm\{\ems\}$ there exists some $V\in\PP$ with $V\sbe U$.  The cardinal invariant \emph{$\pi$-weight}  is defined as
\begin{gather}\label{def:pwx}
\pi w(X) \df \min\set{\card{\PP}\st \PP \text{ is a $\pi$--base for }X}.
\end{gather}

It is easy to see that for a semiregular space $X$,
\begin{gather}\label{pwxb}
\pi w(X)=\pi w_a(\RO(X)) = \pi w_a(\RC(X)).
\end{gather}
Clearly, $\pi w(X) \leq w(X)$, and, as  is well known, the inequality may be strict, even for compact Hausdorff spaces. For example, consider  $\NNNN$ with the discrete topology, and its  Stone-\v{C}ech compactification $\beta\NNNN$. Since $\set{\set{n}\st n \in \NNNN}$ is a $\pi$--base for $\beta\NNNN$, we obtain $\pi w(\beta\NNNN) = \pi w_a(\RC(\beta\NNNN))=  \aleph_0$. On the other hand, it is well known that $w(\b\NNNN)=2^{\aleph_0}$ \cite{E}. The same example shows that $\pi w$ is not isotone, since  $\beta\NNNN \setminus \NNNN \subseteq \beta\NNNN$, and
\begin{gather*}
\aleph_0 = \pi w(\beta\NNNN) \lneq \pi w(\beta\NNNN \setminus \NNNN) = 2^{\aleph_0}.
\end{gather*}
Algebraically, the situation is as follows. Let $B$ be the finite--cofinite algebra over $\NNNN$, and $\overline{B}$ its completion; then, $ \pi w_a(\overline{B}) =\pi w_a(B) = \aleph_0$. Now, $\overline{B}$ is isomorphic to the set algebra $2^{\NNNN}$ which, in turn, is isomorphic to $\RC(\beta\NNNN)$.

In the rest of the section we shall investigate the connections among $w_a, \pi w_a,$ and their corresponding topological notions.

Suppose that $\klam{B,\rho,\BBBB}$ is an LCA. Obviously, \ref{bc3}
implies that $\BBBB$ is dense in $B$. If
 $D$ is a dense subset of $B$, then
$D\cap \BBBB$ is a dense subset of $B$,
since $\BBBB$ is an ideal of $B$.
Furthermore, every base for $\klam{B,\rho,\BBBB}$ is a dense subset of $B$; hence,
\begin{gather}\label{pww}
\pi w_a(B)\le w_a(B,\rho,\BBBB).
\end{gather}

\begin{pro}
Let $\klam{B,\rho,\BBBB}$ be an LCA and $M$ be a subset of $B$. Then
the following conditions are equivalent:
\begin{enumerate}
\item $M$ is a dense subset of $\klam{B,\rho,\BBBB}$.
\item For each $a\in B^+$ there exists $b\in M^+$ such that $b\llx a$.
\item For each $a\in\BBBB^+$, $a=\bv\{b\in M\st b\llx a\}$;
\item For each $a\in B^+$, $a=\bv\{b\in M\st b\llx a\}$.
\end{enumerate}
\end{pro}
\begin{proof}
The implications
\begin{center}
1. $\Iff$ 2.,  3. $\Iff$ 4., and 4. $\Implies$ 1.
\end{center}
 can be easily obtained using (LC3) or \cite[Lemma 4.9.]{kop89}, or the fact that $\BBBB$ is a dense subset of $B$. So we only show 1. $\Implies$ 4. Let $a\in B^+$; then $a=\bv\{b\in M\st b\le a\}$ since $M$ is dense in $B$. Let $a_1\in B$ and $b \leq a_1$ for every $b\in M$ such that $b\llx a$. Assume that $a \not\leq a_1$. Then $a\we a_1^*> 0$. By \ref{bc3} there exists some $c\in M^+$ such that $c\llx a\we a_1^*$, and the density of $M$ implies that there is some  $b\in M^+$ with $b\le c$. Then $b\llx a\we a_1^*$. Thus $b\llx a$; hence, $b\le a_1$ by the definition of $b$. Altogether, we obtain $b\le
a_1\we a_1^*=0$, a contradiction. It follows that $a\le a_1$; therefore, $a=\bv\{b\in M\st b\llx a\}$.
\end{proof}

\begin{defi}\label{pisemsp}
\rm A topological space $(X,\TT)$ is called  $\pi$-{\em semiregular} if  the family $\RO(X)$ is a $\pi$-base for $X$.
\end{defi}

Clearly, every semiregular space is $\pi$-semiregular. The converse is not true. Indeed, the \emph{half--disc topology} from \cite[Example 78]{SS} is a $\pi$-semiregular $T_{2\frac{1}{2}}$-space which is not semiregular.
On the other hand, there exist spaces which are not $\pi$-semiregular: if $X$ is an infinite set with the cofinite topology then $X$ is not a $\pi$-semiregular space since $\RO(X)=\{\ems, X\}$.

The following lemma from \cite{Douwen} is an analogue of  the Alexandroff-Urysohn theorem \cite[Theorem 1.1.15]{E}:

\begin{lm}\label{dislm}{\rm (\cite{Douwen})}
If\/ $\BB$ is a $\pi$-base for a space $X$ then there exists a $\pi$-base $\BB\ap$ of $X$ such that $\BB\ap\sbe\BB$ and $\card{\BB\ap}=\pi w(X)$.
\end{lm}

The next proposition is a generalisation of \eqref{pwxb}.

\begin{pro}\label{piweightpisem}
If $X$ is  $\pi$-semiregular, then  $\pi w(X) =\pi w_a(\RC(X))$.
\end{pro}
\begin{proof}
Since $X$ is $\pi$-semiregular, $\RO(X)$ is a $\pi$-base for $X$. Hence, by Lemma \ref{dislm},   there exists a $\pi$-base $\BB$ of $X$ such that $\BB\sbe \RO(X)$ and $\card{\BB}=\pi w(X)$; obviously, $\BB$ is a dense subset of $\RO(X)$ as well. Hence,  $\pi w(X)\ge\pi w_a(\RO(X))$, and, clearly,  $\pi w(X) \le\pi
w_a(\RO(X)) = \pi w_a(\RC(X))$.
\end{proof}

\begin{pro}\label{prorho}
Let $A$ be an infinite Boolean algebra. Then there exists a normal contact relation $\rho$ on $A$ such that
$w_a(\klam{A,\rho})=\pi w_a(A)$ and $\klam{A,\rho}_S$ is a base for $\klam{A,\rho}$.
\end{pro}
\begin{proof}
There exists a dense subset $D$ of $A$ with $\card{D}=\pi w_a(A)$. Note that $\pi w_a(A)\ge\aleph_0$.
Let $B$ be the Boolean subalgebra of $A$ generated by $D$.
Now, Proposition \ref{nuldimlm} implies that there exists a normal contact relation $\rho$ on $A$ such that $B\df\klam{A,\rho}_S$ is a base for $\klam{A,\rho}$ and
$w_a(\klam{A,\rho})=\card{B}$.  Since $|B|=|D|$, we obtain  $w_a(\klam{A,\rho})=\pi w_a(A)$.
\end{proof}

\begin{theorem}\label{pon2}
Let $X$ be a $\pi$-semiregular space and $\pi w(X)\ge\aleph_0$. Then  there exists a zero-dimensional compact Hausdorff space $Y$ with $w(Y)=\pi w(X)$ such that the Boolean algebras $\RC(X)$ and $\RC(Y)$ are isomorphic.
\end{theorem}
\begin{proof}
Set $\tau\df\pi w(X)$ and $A \df \RC(X)$. Then, by Proposition \ref{piweightpisem}, $\pi
w_a(A)=\tau$. Now, by Proposition \ref{prorho}, there exists a normal contact relation $\rho$ on $A$ such that $w_a(\klam{A,\rho})=\tau$ and $\klam{A,\rho}_S$ is a base for $\klam{A,\rho}$. Using Theorems \ref{nuldim} and \ref{weightpiw}, we see that  $Y \df \LAM^a(\klam{A,\rho})$ is a zero-dimensional compact Hausdorff space with $w(Y)=\tau$.  Finally, by de Vries' duality theorem, $\RC(Y)$ is isomorphic to $A$, i.e. to $\RC(X)$.
 \end{proof}

  Theorem \ref{pon2} is not true for general spaces with infinite $\pi$-weight. Indeed, let $X$ be countably infinite with the cofinite topology; then, $\pi w(X) = \aleph_0$, and $\RC(X) = \set{\emptyset,X}$. On the other hand, if $Y$ is a zero-dimensional compact Hausdorff space with $\RC(Y) = \set{\emptyset,Y}$ then $1=w(Y)<\aleph_0=\pi w(X)$.


\begin{thebibliography}{99}
{\small

\bibitem{AHS}
  J. Ad\'amek, H. Herrlich, G. E. Strecker,
  \newblock  {\em Abstract and Concrete Categories},
   \newblock http://katmat.math.uni-bremen.de/acc, 2004.


\bibitem{Alex}
 P. S. Alexandroff,
\newblock  {\em Outline of Set Theory and General Topology},
\newblock  Nauka, Moskva, 1977  (In Russian).


\bibitem{AP}
 A. V. Arhangel'skii and V. I. Ponomarev,
\newblock  {\em Fundamentals of General Topology: Problems and Exercises},
\newblock Reidel, Dordrecht, 1984. Originally published by Izdatelstvo Nauka, Moscow, 1974.




\bibitem{Biacino-and-Gerla-1996}
 L. Biacino, G. Gerla,
\newblock {\em Connection structures: Grzegorczyk's and Whitehead's definition of point},
\newblock  Notre Dame Journal of Formal Logic,  37 (1996), 431--439.

\bibitem{bir48}
  G. Birkhoff,
  \newblock  {\em Lattice Theory},
   \newblock Providence, Rhode Island, 1967.


\bibitem{Celani}
S. Celani,
\newblock  {\em Quasi-modal algebras},
\newblock Math. Bohem., 126 (2001), 721--736.



\bibitem{D-a0903-2593}
G. Dimov,
\newblock  {\em A  de Vries-type duality theorem for locally compact spaces -- II},
\newblock arXiv:0903.2593v4, 1-37.






\bibitem{D-AMH1-10}
G. Dimov,
\newblock  {\em A  de Vries-type duality theorem for the category of locally compact spaces and continuous maps -- I},
\newblock Acta Math. Hungarica, 129 (2010), 314--349.

\bibitem{D-AMH2-10}
G. Dimov,
\newblock  {\em A  de Vries-type duality theorem for the category of locally compact spaces and continuous maps -- II},
\newblock Acta Math. Hungarica, 130 (2011), 50--77.

\bibitem{D-PMD12}
G. Dimov,
\newblock  {\em Some generalizations of the Stone Duality Theorem},
\newblock Publicationes Mathematicae Debrecen, 80 (2012), 255--293.



\bibitem{DDV}
G. Dimov, E. Ivanova-Dimova, D. Vakarelov,
\newblock  {\em A generalization of the Stone Duality Theorem},
\newblock Topology Appl. 221 (2017), 237--261.


\bibitem{DV1}
G. Dimov, D. Vakarelov,
\newblock  {\em Contact algebras and region-based theory of space: a proximity approach -- I},
\newblock Fundamenta Informaticae 74(2-3) (2006), 209--249.




\bibitem{DV-LN06}
G. Dimov, D. Vakarelov,
\newblock  {\em Topological representation of precontact algebras},
\newblock In W. MacCaull, M. Winter, and I. D\"{u}ntsch, editors, {\em Relation Methods in Computer Science}, Lecture Notes in Computer Science 3929. Springer-Verlag (Berlin
Heidelberg), 2006, 1--16.


\bibitem{DV3}
G. Dimov, D. Vakarelov,
\newblock  {\em Topological representation of precontact algebras and a connected version of the Stone duality theorem -- I},
\newblock Topo\-logy Appl. 227 (2017), 64--101.

\bibitem{Douwen}
Eric K. van Douwen,
\newblock  {\em Cardinal functions on Boolean spaces},
\newblock In J. Donald Monk
and R. Bonnet, editors, {\em Handbook of Boolean algebras, vol. 2},  417--467,
North-Holland, Amsterdam - New York - Oxford - Tokyo, 1989.


\bibitem{DUV}
I. D\"{u}ntsch, D. Vakarelov,
\newblock  {\em Region-based theory of discrete spaces: A proximity approach},
\newblock Annals of Mathematics and Artificial Intelligence, 49 (2007), 5--14.


\bibitem{DW}
I. D\"{u}ntsch, M. Winter,
\newblock  {\em A representation theorem for Boolean contact algebras},
\newblock Theoretical Computer Science (B), 347 (2005), 498--512.


\bibitem{E}
 R. Engelking,
\newblock  {\em General Topology},
\newblock PWN, Warszawa, 1977.



\bibitem{eng_dim}
 R. Engelking,
\newblock  {\em Dimension Theory},
\newblock North Holland, Amsterdam, 1978.

\bibitem{F}
 V. V. Fedorchuk,
\newblock  {\em Boolean $\d$-algebras and quasi-open mappings},
\newblock  Sibirsk. Mat. \v{Z}. 14 (5) (1973), 1088--1099; English translation: Siberian Math. J.
14 (1973), 759-767 (1974).




\bibitem{kop89}
S. Koppelberg
\newblock  {\em Handbook on Boolean Algebras, vol. 1: General Theory of Boolean Algebras},
\newblock North Holland,  1989.

\bibitem{LE}
  S. Leader,
\newblock {\em Local proximity spaces},
\newblock  Math. Annalen 169 (1967), 275--281.



\bibitem{monk14}
 J. D. Monk,
\newblock  {\em Cardinal Invariants on Boolean Algebras},
\newblock    Progress in Mathematics 142, Springer Verlag, 2014.



\bibitem{PW}
J. R. Porter and R. G. Woods
\newblock {\em Extensions and Absolutes of Hausdorff Spaces},
\newblock Springer-Verlag, New York, 1988.


\bibitem{Roeper}
P. Roeper,
\newblock  {\em Region-based topology},
\newblock  Journal of Philosophical Logic 26 (1997), 251--309.




\bibitem{rub76}
M. Rubin,
\newblock {\em The theory of Boolean algebras with a distinguished subalgebra is undecidable},
\newblock Annales scientifiques de l'Universite de Clermont. Mathematiques, 60(13) (1976), 19--34.


\bibitem{SS}
L. A. Steen, J. A. Seebach,
\newblock  {\em Counterexamples in Topology},
\newblock Springer-Verlag, 1978.



\bibitem{Stone}
 M. H. Stone,
\newblock  {\em Applications of the theory of Boolean rings to general topology},
\newblock  Trans. Amer. Math. Soc. 41 (1937), 375--481.




\bibitem{VDDB}
 D. Vakarelov, G. Dimov, I. D{\"u}ntsch, B. Bennett,
\newblock  {\em A proximity approach to some region-based theories of
space},
\newblock  J. Applied Non-Classical Logics 12 (3-4) (2002),
527-559.


\bibitem{deV}
H. de Vries,
\newblock  {\em Compact Spaces and Compactifications, an Algebraic Approach},
\newblock Van Gorcum, The Netherlands, 1962.
}
\end{thebibliography}

\baselineskip = 1.2\normalbaselineskip

\end{document}